\documentclass[11pt]{amsart}
\usepackage{latexsym}
\usepackage{amscd}
\usepackage{amssymb}
\usepackage[all,cmtip]{xy}
\usepackage{tikz}
\usepackage{tikz-cd}
\usepackage{amsmath, pb-diagram}
\usepackage{enumerate}
\DeclareMathAlphabet{\eusm}{OT1}{eusm}{m}{n}
\newtheorem{theorem}{Theorem}[section]
\newtheorem{prop}[theorem]{Proposition}
\newtheorem{defn}[theorem]{Definition}
\newtheorem{cor}[theorem]{Corollary}
\newtheorem{lem}[theorem]{Lemma}

\newtheorem{expl}[theorem]{Example}
\newtheorem{expls}[theorem]{Examples}
\newtheorem{rem}[theorem]{Remark}

\DeclareMathOperator{\Hom}{Hom}
\DeclareMathOperator{\Ext}{Ext}
\DeclareMathOperator{\Coker}{Coker}
\DeclareMathOperator{\Ker}{Ker}
\DeclareMathOperator{\dom}{dom}
\DeclareMathOperator{\Img}{Im}

\DeclareMathOperator{\Inj}{Inj}
\DeclareMathOperator{\PInj}{PInj}
\DeclareMathOperator{\FInj}{FInj}
\DeclareMathOperator{\cf}{cf}
\DeclareMathOperator{\rest}{\upharpoonright}
\DeclareMathOperator{\End}{End}
 \newcommand{\ModR}{\textrm{Mod-}R}

\begin{document}
 \title{Ziegler Partial Morphisms in additive exact categories}
 \author[Cort\'es-Izurdiaga]{Manuel Cort\'es-Izurdiaga}
 \address{Departamento de Matem\'aticas, Universidad de Almeria, E-04071,
   Almeria, Spain} \email{mizurdia@ual.es} \thanks{The first author is partially supported by the Spanish Government under
   grants MTM2016-77445-P and MTM2017-86987-P which include FEDER funds of the EU}
 \author[Guil Asensio]{Pedro A. Guil Asensio}
 \address{Departamento de Matem\'aticas, Universidad
   de Murcia, Murcia, 30100, Spain} \email{paguil@um.es} \thanks{The
   second author is partially supported by the Spanish Government under
   grant MTM2016-77445-P which includes FEDER funds of the EU, and by Fundaci\'on S\'eneca of
   Murcia under grant
  19880/GERM/15}
 \author[Kalebo\~{g}az]{Berke Kalebo\~{g}az}
 \address{Department of Mathematics, Hacettepe University, Ankara,
   Turkey} \email{bkuru@hacettepe.edu.tr}
 \author[Srivastava]{Ashish K. Srivastava} \address{Department of
   Mathematics and Statistics, Saint Louis University, St.  Louis,
   MO-63103, USA} \email{ashish.srivastava@slu.edu} \thanks{The fourth author is partially supported by a grant from Simons
   Foundation (grant number 426367).}
 \maketitle

\begin{abstract}
  We develop a general theory of partial morphisms in
  additive exact categories which extends the model theoretic notion introduced by Ziegler in 
 the particular case of pure-exact sequences in the category of modules over a ring. We relate
  partial morphisms with (co-)phantom morphisms and injective approximations and study the
  existence of such approximations in these exact categories.
\end{abstract}

\bigskip

\bigskip

\section*{Introduction}

\bigskip

\noindent We introduce and develop a general theory of partial
morphisms in arbitrary additive exact categories, in the sense of Quillen. Exact categories are a natural generalization of abelian categories, and they play a quite useful role in several areas, like Representation Theory, Algebraic Geometry, Algebraic Analysis and Algebraic $K$-Theory. The main reason behind their usefulness is that they are applicable in many situations in which the classical theory of abelian categories does not apply, for instance, in the study of filtered objects and tilting theory.

 Partial morphisms were
introduced by Ziegler in \cite{Ziegler} in his study of Model
Theory of Modules, in order to prove the existence of pure-injective
envelopes. Recall that a short exact sequence of right modules is called pure if it remains exact upon tensoring by any left module (equivalently, when it is a direct limit of splitting short exact sequences). And therefore, purity reflects all decomposition properties of modules into direct summands. Ziegler realized that pure-injective modules (i.e., those modules which are injective with respect to pure-exact sequences) also extend other types of morphisms and called those morphisms as {\em partial morphisms}. These partial morphisms were central to giving a right pure version of the notion of essential monomorphisms in the category of modules.

 This concept was later stated in an algebraic language by
Monari Martinez \cite{Monari} in terms of systems of linear equations. Namely, she gave a matrix-theoretic
reformulation of it. Given a ring $R$ (not necessarily commutative), a
submodule $K$ of a right $R$-module $M$ and a right $R$-module $N$, a
homomorphism $f:K\rightarrow N$ is called a partial morphism from $M$
to $N$ if whenever we have
a system of linear equations $$\begin{bmatrix} x_1& . & . & . &
  x_m
\end{bmatrix} A=\begin{bmatrix} b_1 & . & . & . & b_n
\end{bmatrix}
$$ with $A\in \mathbb M_{m\times n}(R)$ and $b_1, \ldots, b_n \in K$,
which is solvable in $M$, then the system
$$\begin{bmatrix} x_1 & . & . & . & x_m
\end{bmatrix} A=\begin{bmatrix} f(b_1) & . & . & . & f(b_n)
\end{bmatrix}
$$ is also solvable in $N$. 

However, the above algebraic translation of the notion of partial morphisms 
does not shed much light about their role in the categorical study of purity.
In the present paper, we give a categorical definition
of this concept which can be stated in any additive exact category
$(\mathcal A;\mathcal E)$ (i.e., an additive category $\mathcal A$
with a distinguished class $\mathcal E$ of kernel-cokernel pairs which
play the role of short exact sequences). This definition reduces to
the original one introduced by Ziegler in the specific case of the
pure-exact structure in $\ModR$ consisting of all pure-exact sequences and it explains the importance of partial morphisms in a much more transparent way: a homomorphism $f:K\rightarrow N$ is partial respect to the inclusion $u$ of  $K$ in a module $M$ if and only if the induced morphism ${\rm Ext}^1(-,f)$ transforms $u$ in a pure monomorphism (see Theorem \ref{p:CharacterizationZieglerPartial}). 

 As Ziegler himself observed for the particular case of modules, this notion of partial morphisms allows us to introduce the definition of {\em small} morphisms in exact categories. And it is therefore related to the existence of approximations of modules. We explain how this idea of approximation is interrelated with others used in the literature. Namely, we show, in Theorem \ref{t:InjectiveHulls}, that this idea of approximation in terms of small extensions is equivalent to the one introduced by Enochs of monomorphic envelopes in the category of modules \cite{Enochs}, and to the classical one defined in terms of essential or pure-essential subobjects. Then we prove the existence of enough injectives in certain additive exact categories (see Theorem 4.4, which is one of the main results of this paper), and the existence of injective approximations (in the sense of small morphism mentioned before) in certain exact structures of abelian categories (see Theorem \ref{t:ExistenceHulls}). 

 As an application of our results, we are able to recover several well-known classical results such as the existence of injective hulls in Grothendieck categories, and the existence of pure-injective hulls in finitely accessible additive categories. But, moreover, our theory also includes the known results about approximations relative to a class of modules \cite{GobelTrlifaj}.  The key idea is that, under quite general assumptions, finding preenvelopes in an exact category with respect to a class $\mathcal X$ of objects is equivalent to show that there exists enough $E^\mathcal{X}$-injectives, where $E^\mathcal{X}$ is the exact structure consisting of all conflations  $A \rightarrow B \rightarrow C$ which are $\Hom(-,X)$-exact for every $X\in\mathcal{X}$. Applying these arguments to Theorem 4.4, we deduce Corollary~5.4, a result which recovers [17, Theorem 2.13(4)]. This is probably the most general known result of existence of (pre-)envelopes in exact categories. The same ideas are later applied to Theorem 4.4 and Theorem 4.11 to prove our Theorem 5.6, which covers all known results of approximations relative to cotorsion pairs in Grothendieck categories. We also relate all these constructions with the recent theory of approximations of objects by ideals of morphisms introduced in \cite{FuGuilHerzogTorrecillas} (see Corollary~\ref{c:Cophantom}). In conclusion, we provide a quite general theory in which most known results of approximations of objects in exact categories are deduced as consequences of our general results, and we also explain how they are interrelated with each other.

Let us briefly outline the structure of this paper. After recalling
some terminology and preliminary facts, we define, in Section 2,
partial morphisms with respect to an additive exact substructure
$\mathcal F$ (the $\mathcal F$-partial morphisms) of an exact
structure $\mathcal E$ in an additive category $\mathcal A$ (see
Section 1). In order to do it, we first need to give a categorical
characterization of partial morphisms relative to the pure-exact
structure in the a module category (see in Theorem
\ref{p:CharacterizationZieglerPartial}).  This characterization is
obtained in terms of pushouts and thus, it allows us to extend the
notion of partial morphism to the wider framework of additive exact
categories. Then, we study the properties of $\mathcal F$-partial
morphisms and extend several of the results proved by Ziegler to this
new setting. It is especially relevant that, as in the case of pure-injective modules, $\mathcal F$-partial morphisms can be used to
characterize $\mathcal F$-injective objects. More precisely, we prove,
in Theorem \ref{t:FInjectivePartial}, that an object $E$ in
$\mathcal A$ is $\mathcal F$-injective if and only if any
$\mathcal F$-partial morphism $f$ from an object $X$ to $E$ extends to
a morphism $g:X \rightarrow E$. This extends the corresponding theorem
for pure-injective modules proved by Ziegler \cite[Theorem 1.1,
Corollary 3.3]{Ziegler}.  Other advantage of our definition in terms
of pushouts is that it allows to relate partial and phantom morphisms
(see \cite{FuGuilHerzogTorrecillas} for a definition and main
properties of these phantom morphisms).

In Section 3, we introduce small subobjects using partial
morphisms. Then, we can define when an inclusion $u:U \rightarrow E$,
with $E$ injective, is small; which in turn is related to the notion
of injective approximations in the category. Recall that an injective
hull in an abelian category $\mathcal B$ is an essential inclusion
$u:U \rightarrow E$ with $E$ injective, in the sense that
$U \cap V \neq 0$ for each non-zero subobject $V$ of $E$. It is well
known that the injective hull $u$ is an injective envelope too, in the
sense that any endomorphism $f:E \rightarrow E$ such that $fu=f$ is an
automorphism. We compare these notions of small injective extensions
with that defined in terms of partial morphisms and prove, in Theorem
\ref{t:InjectiveHulls}, that for nice categories, all of them are
equivalent.

Our discussion of injective approximations in exact categories leads
us in Section 4 to study when these approximations do exist. The
solution to this problem requires answering the following two
questions:
\begin{enumerate}[(1)] 
\item Do there exist enough injectives in the category (in
the sense that each object can be embedded in an injective one)? 

\item Assuming the category has enough injectives, can these embeddings be
chosen small? 
\end{enumerate}

\noindent In Theorem \ref{t:ExistenceInjectives}, we prove that
Question 1 has a positive answer for additive exact categories
satisfying a generalization of Baer's lemma. And, in Theorem
\ref{t:ExistenceHulls}, we describe a construction of small injective
approximations for exact substructures of abelian categories. 

We end the paper with Section 5, in which we apply our results to study the approximation by objects in exact, Grothendieck and finitely accessible additive categories. In Corollary \ref{fp-inj} we prove that every module has an fp-injective preenvelope. In Corollary \ref{Grothendieck} we prove that every object in the Grothendieck category has an injective hull. In Corollary \ref{pinj} we prove that every object in abelian finitely accessible additive category has a pure-injective hull.

\bigskip

\section{Preliminaries}

\bigskip

\noindent Given a set $A$, we shall denote by $|A|$ its
cardinality. Given a map $f:A \rightarrow B$ and $C$ a subset of $A$,
we shall denote by $f \rest C$ the restriction. All our categories are
additive (that is, they have finite direct products and an abelian
group structure on each of their hom-sets which is compatible with composition).
Let us fix some notations about subobjects in a category.
\begin{defn}
Let $\mathcal A$ be a category and $A$ an object of $\mathcal A$.
\begin{enumerate}
\item Two monomorphisms $u:U \rightarrow A$ and $v:V \rightarrow A$ are equivalent if there exists an isomorphism $w:V \rightarrow U$ such that $uw=v$. An equivalence class of monomorphisms under this equivalence relation is a subobject of $A$. Given a representative $u:U \rightarrow A$ of this equivalence class, we shall simply say that $U$ is a subobject of $A$, we shall write $U \leq A$ and the monomorphism $u$ will be called an inclusion of $U$ in $A$.

\item Given two subobjects $U$ and $V$ of $A$, we shall write $U \subseteq V$ if $U \leq V$ and there exist inclusions $u:U \rightarrow A$, $v:V \rightarrow A$ and $w:U \rightarrow V$ such that $vw=u$.
\end{enumerate}
\end{defn}

\noindent By a {\em kernel-cokernel} pair in $\mathcal A$ we mean a pair of
composable morphisms
\begin{displaymath}
  \begin{tikzcd}
    B \arrow{r}{i} & C \arrow{r}{p} &A
  \end{tikzcd}
\end{displaymath}
such that $i$ is a kernel of $p$ and $p$ is a cokernel of $i$.

\bigskip

\noindent The following lemma is straightforward but very useful, so
we state it without any proof.

\begin{lem}\label{l:DiagramLemma}
  Let $\mathcal A$ be a category. Consider the following commutative
  diagram
  \begin{displaymath}
    \begin{tikzcd}
      B \arrow{r}{i} \arrow{d}{\varphi_1}& C \arrow{r}{p}
      \arrow{d}{\varphi_2} & A \arrow{d}{\varphi_3}\\
      B' \arrow{r}{i'} & C' \arrow{r}{p'} & A'
    \end{tikzcd}
  \end{displaymath}
  in which $p$ is a cokernel of $i$ and $i'$ is a kernel of $p'$. Then
  the following assertions are equivalent:
  \begin{enumerate}
  \item There exists $\alpha \colon A \rightarrow C'$ such that
    $p'\alpha = \varphi_3$.
  \item There exists $\beta \colon C \rightarrow B'$ such that
    $\beta i=\varphi_1$.
  \end{enumerate}
\end{lem}

\noindent Given two morphisms $f:K \rightarrow M$ and $g:K \rightarrow
N$ in any category $\mathcal A$, the pushout diagram of $f$ and $g$ consists of an object $P$ and morphisms $i_1:M\rightarrow P$ and $i_2:N\rightarrow P$ such that the following diagram commutes
\begin{displaymath}
  \begin{tikzcd}
    K \arrow{r}{f} \arrow{d}{g} & M \arrow{d}{i_1}\\
    N \arrow{r}{i_2} & P
  \end{tikzcd}
\end{displaymath}
and the triple $(P, i_1, i_2)$ is universal in the sense that whenever $(Q, j_1, j_2)$ is any other triple making the above diagram commutative, then there exists a unique morphism $\varphi: P\rightarrow Q$ such that $j_1=\varphi i_1$ and $j_2=\varphi i_2$.

We recall some well-known facts about pushouts, which shall be used
throughout the paper.

\begin{lem}\label{l:PushoutCokernel}
  Let $\mathcal A$ be a category. Consider the following pushout
  diagram:
  \begin{displaymath}
    \begin{tikzcd}
      M \arrow{r}{f} \arrow{d}{g} & N \arrow{d}{\overline g}\\
      L \arrow{r}{\overline f} & P
    \end{tikzcd}
  \end{displaymath}
  Then:
  \begin{enumerate}
  \item The morphism $\overline g$ is a split monomorphism if and only
    if there exists $h:L \rightarrow N$ with $hg=f$.
  \item If $f$ has a cokernel $c:N \rightarrow C$, then the unique
    morphism $c':P \rightarrow C$ satisfying $c'\overline g=c$ and $c'\overline f = 0$ is a
    cokernel of $\overline f$.

  \item If $\overline f$ has a cokernel $c'$, then $c'\overline g$ is
    a cokernel of $f$.
  \end{enumerate}
\end{lem}

\noindent For exact categories, we mostly rely on \cite{Buhler} but we
use some terminologies of \cite{Keller} as well. Let $\mathcal A$ be a
category. An \textit{exact structure} on $\mathcal A$ is a family
$\mathcal E$ of distinguished kernel-cokernel pairs satisfying axioms
[E0] - [E2] and [E0$^{\textrm{op}}$] - [E2$^{\textrm{op}}$] from
\cite{Buhler}. We shall denote by $(\mathcal A;\mathcal E)$ the exact
category and elements in $\mathcal E$ will be called {\it
  conflations}. The kernel of a conflation is called {\it inflation}
and the cokernel of a conflation is called {\it deflation}. An
\textit{admissible subobject} of an object $A$ is a subobject $U$ of
$A$ such that one (and then any) inclusion
$i:U \rightarrow A$ is an inflation. The main example of an exact
category is an abelian category with the exact structure formed by all
kernel-cokernel pairs. We shall call this exact structure the
\textit{abelian exact structure}.

Let $(\mathcal A; \mathcal E)$ be an exact category. Given $E$ an object and $u:K \rightarrow A$ an inflation, we say that
$E$ is \textit{$u$-injective} (or injective with respect to $u$) if
for each morphism $f:K \rightarrow E$, there exists a
$g:A \rightarrow E$ with $gu=f$. If $\mathcal H$ is a class of
inflations, we say that the object $E$ is
$\mathcal H$-\textit{injective} if it is $u$-injective for each object
$u \in \mathcal H$. If $X$ is another object, we say that $A$ is
$X$-injective if it is injective with respect to each inflation
$u:K \rightarrow X$. Finally, we say that $E$ is injective if it is
injective with respect to each inflation. This is equivalent to the
functor $\Hom_\mathcal{A}(-,E)$, from $\mathcal A$ to the category
$\mathbf{Ab}$ of abelian groups, carrying inflations to
epimorphisms. We shall say that
$\mathcal A$ has enough injective objects if for each object $A$ in
$\mathcal A$, there exists an inflation $A \rightarrow E$ with $E$ an 
injective object in $\mathcal A$. The notions about projectivity in exact categories are
defined dually.

We shall use the following result about relative injective
objects, which is well known for the abelian exact structure of an
abelian category, and for the pure-exact structure in module
categories.

\begin{lem}\label{l:AInjective}
  Let $(\mathcal A; \mathcal E)$ be an exact category. Let $M$ be an object of $\mathcal A$ and
  \begin{displaymath}
    \begin{tikzcd}
      A \arrow{r}{i} & B \arrow{r}{p} & C
    \end{tikzcd}
  \end{displaymath}
  be a conflation. If $M$ is $B$-injective, then $M$ is both
  $A$-injective and $C$-injective.
\end{lem}

\begin{proof}
  Given an inflation $u:K \rightarrow A$ and $f:K \rightarrow M$, $iu$
  is an inflation so that there exists $g:B \rightarrow M$ with
  $giu= f$. Then $M$ is $A$-injective.

  In order to see that $M$ is $C$-injective, take $u:K \rightarrow C$
  an inflation and $f:K \rightarrow M$. Taking pullback of $u$ along
  $p$ we get the following commutative diagram
  \begin{displaymath}
    \begin{tikzcd}
      P \arrow{r}{\overline p} \arrow{d}{\overline u} & K
      \arrow{d}{u}\\
      B \arrow{r}{p} & C \\
    \end{tikzcd}
  \end{displaymath}
  in which $\overline u$ is an inflation by \cite[Proposition
  2.15]{Buhler}. Let $\overline i:\overline A \rightarrow P$ be a kernel of $\overline p$. Using the universal property of the pullback, we can construct a commutative diagram with a conflation in each row,
  \begin{displaymath}
    \begin{tikzcd}
      \overline A \arrow{r}{\overline i} \arrow{d}{w} & P \arrow{r}{\overline p} \arrow{d}{\overline u} & K 
      \arrow{d}{u}\\
     A \arrow{r}{i} & B \arrow{r}{p} & C
    \end{tikzcd},
  \end{displaymath}  
  with $w$ an isomorphism. Now, using that $M$ is $B$-injective, there exists
  $g':B \rightarrow M$ with $g'\overline u=f \overline p$. Notice
  that $g'iw=g' \overline u \overline i = f \overline p \overline i=0$ and, since $w$ is isomorphism, $g'i=0$, so that there exists $g:C \rightarrow M$ with
  $gp=g'$. Then $gu\overline p=f \overline p$ and, since $\overline p$
  is an epimorphism, $gu=f$ as well. Then $M$ is $C$-injective.
\end{proof}

Let $(\mathcal A; \mathcal E)$ be an exact category. Given two objects
$A,B$ in $\mathcal A$, we shall denote by $\Ext(A,B)$ the abelian
group whose elements are the isomorphism classes of all conflations of
the form
\begin{displaymath}
  \begin{tikzcd}
    B \arrow{r}{i} & C \arrow{r}{p} &A
  \end{tikzcd}
\end{displaymath}
equipped with the {\em Baer sum} operation. Given any morphism
$g\colon B \rightarrow X$, we can define a morphism
$\Ext(A,g):\Ext(A,B) \rightarrow \Ext(A,X)$ as follows: for any
conflation
\begin{displaymath}
  \begin{tikzcd}
    \eta: & B \arrow{r}{i} & C \arrow{r}{p} &A
  \end{tikzcd}
\end{displaymath}
we take the pushout of $i$ along $g$ to get a commutative diagram
\begin{displaymath}
  \begin{tikzcd}
    \eta: & B \arrow{r}{i} \arrow{d}{g} & C \arrow{d}{\overline g}
    \arrow{r}{p}
    & A \arrow[equal]{d}\\
    \eta': & X \arrow{r}{\overline i} & P \arrow{r}{\overline p} & A
  \end{tikzcd}
\end{displaymath}
in which $\eta'$ is a conflation by \cite[Proposition
2.12]{Buhler}. Then define $\Ext(A,g)(\eta)=\eta'$. Similarly, we can
define, using pullbacks, $\Ext(f,B):\Ext(A,B) \rightarrow \Ext(X,B)$
for each morphism $f:X \rightarrow A$.

If we fix the objects $A$, $B$ and $X$, $\Ext(A,-)$ actually defines a
map from $\Hom_{\mathcal A}(B,X)$ to
$\Hom_{\mathbb Z}\big(\Ext(A,B),\Ext(A,X)\big)$ which actually is a
morphism of abelian groups. Similarly, we obtain a morphism of abelian
groups $\Ext(-,B)$ from $\Hom_{\mathcal A}(X,A)$ to
$\Hom_{\mathbb Z}\big(\Ext(A,B),\Ext(X,B)\big)$.

An \textit{exact substructure} $\mathcal F$ of $\mathcal E$
is an exact structure on $\mathcal A$ such that each conflation in
$\mathcal F$ (which we shall call $\mathcal F$-conflations) is a
conflation in $\mathcal E$. Inflations, deflations, admissible subobjects and
injective objects with respect to $\mathcal F$ will be called
$\mathcal F$-inflations, $\mathcal F$-deflations,
$\mathcal F$-admissible and $\mathcal F$-injective objects,
respectively. Moreover, if $(\mathcal A; \mathcal F)$ has enough
injective objects, we shall say that $\mathcal A$ has enough
$\mathcal F$-injective objects.

Given a class $\mathcal X$ of objects we shall denote by
$\mathcal E_\mathcal{X}$, the class of all
$\Hom_{\mathcal A}(\mathcal X,-)$-exact conflations, i.e., those
conflations
\begin{displaymath}
  \begin{tikzcd}
    A \arrow{r} & C \arrow{r} & B
  \end{tikzcd}
\end{displaymath}
such that
\begin{displaymath}
  \begin{tikzcd}
    \Hom_{\mathcal A}(X,A) \arrow{r} & \Hom_{\mathcal A}(X,C)
    \arrow{r} & \Hom_{\mathcal A}(X,B)
  \end{tikzcd}
\end{displaymath}
is a short exact sequence in the category of abelian groups for each
$X \in \mathcal X$. Dually, we define $\mathcal E^{\mathcal X}$ to be the class of all $\Hom_{\mathcal A}(-,\mathcal X)$-exact conflations, that is, those conflations
\begin{displaymath}
\begin{tikzcd}
A \arrow{r} & B \arrow{r} & C 
\end{tikzcd}
\end{displaymath}
such that
\begin{displaymath}
\begin{tikzcd}
\Hom_{\mathcal A}(B,X) \arrow{r} & \Hom_{\mathcal A}(C,X) \arrow{r} & \Hom_{\mathcal A}(A,X)
\end{tikzcd}
\end{displaymath}
is a short exact sequence in the category of abelian groups for each $X \in \mathcal X$.
Both $\mathcal E_{\mathcal X}$ and $\mathcal E^{\mathcal X}$ are additive
exact substructure of $\mathcal E$, \cite[Exercise 5.6]{Buhler}.

Using Lemma \ref{l:DiagramLemma} we get a similar description of
$\mathcal E_{\mathcal X}$-conflations to that of pure-exact sequences
in module categories (see \cite[34.5]{Wisbauer}). The result can be
easily dualized for $\mathcal E^{\mathcal X}$-conflations.

\begin{lem}
  Let $(\mathcal A; \mathcal E)$ be an exact category, $\mathcal X$ be
  a class of objects and
  \begin{displaymath}
    \begin{tikzcd}
      \eta: & A \arrow{r}{i} & B \arrow{r}{j} & C
    \end{tikzcd}
  \end{displaymath}
  be a conflation.
  \begin{enumerate}
  \item If $\eta \in \mathcal E_{\mathcal X}$ then for each morphism
    $f \colon M \rightarrow N$ with $\Coker f \in \mathcal X$ and
    commutative diagram
    \begin{displaymath}
      \begin{tikzcd}
        M \arrow{r}{f} \arrow{d}{\varphi_1} & N \arrow{d}{\varphi_2}\\
        A \arrow{r}{i} & B
      \end{tikzcd}
    \end{displaymath}
    there exists $\beta:N \rightarrow A$ such that
    $\beta f = \varphi_1$.
  \item If there exist enough $\mathcal E_{\mathcal X}$-projective
    objects, and $\eta$ satisfies (1), then
    $\eta \in \mathcal E_{\mathcal X}$.

  \item If $\eta \in \mathcal E^{\mathcal X}$ then for each morphism
    $f:M \rightarrow N$ with $\Ker f \in \mathcal X$ and commutative
    diagram
    \begin{displaymath}
      \begin{tikzcd}
        B \arrow{r}{j} \arrow{d}{\psi_1}& C \arrow{d}{\psi_2}\\
        M \arrow{r}{f} & N
      \end{tikzcd}
    \end{displaymath}
    there exists $\alpha:C \rightarrow M$ with $f\alpha = \psi_2$.
  \item If $\mathcal E^{\mathcal X}$ has enough injective objects and
    $\eta$ satisfies (3), then $\eta \in \mathcal E^{\mathcal X}$.
  \end{enumerate}
\end{lem}

\begin{proof}
  (1) Follows from Lemma \ref{l:DiagramLemma}.

  (2) Take $X \in \mathcal X$ and $\varphi_3:X \rightarrow C$, a
  morphism. Let
  \begin{displaymath}
    \begin{tikzcd}
       K \arrow{r}{i} & P \arrow{r}{p} & X
    \end{tikzcd}
  \end{displaymath}
  be an $\mathcal E_\mathcal{X}$-conflation with $P$ being an
  $\mathcal E_{\mathcal X}$-projective object. Using the projectivity
  of $P$ we can construct a commutative diagram
  \begin{displaymath}
    \begin{tikzcd}
      K \arrow{r}{i} \arrow{d}{\varphi_1}& P \arrow{r}{p}
      \arrow{d}{\varphi_2} & X \arrow{d}{\varphi_3}\\
      A \arrow{r}{i'} & B \arrow{r}{p'} & C
    \end{tikzcd}
  \end{displaymath}
  Then the result follows from (1) and Lemma \ref{l:DiagramLemma}.

  (3) and (4) are proved dually.
\end{proof}

Given a class of objects $\mathcal X$ in $\mathcal A$, we define the right and left perpendicular classes to $\mathcal X$, $\mathcal X^\perp$ and ${^\perp}{\mathcal X}$, by
\begin{displaymath}
\mathcal X^{\perp} = \{Y \in \mathcal A \mid \Ext(X,Y)=0, \forall X \in \mathcal X\}
\end{displaymath}
and
\begin{displaymath}
{^\perp}{\mathcal X} = \{Y \in \mathcal A \mid \Ext(X,Y)=0, \forall X \in \mathcal X\}
\end{displaymath}
respectively. A cotorsion pair in $\mathcal A$ is a pair of classes $(\mathcal B,
\mathcal C)$ of objects of
$\mathcal A$, such that $\mathcal B = {^\perp}{\mathcal C}$ and $\mathcal C = \mathcal B^{\perp}$. The
cotorsion pair is said to be complete if for each object $A$ of
$\mathcal A$ there exist conflations
\begin{displaymath}
  \begin{tikzcd}
    A \arrow{r} & C_1 \arrow{r} & B_1
  \end{tikzcd}
\end{displaymath}
and
\begin{displaymath}
  \begin{tikzcd}
    C_2 \arrow{r} & B_2 \arrow{r} & A
  \end{tikzcd}
\end{displaymath}
with $B_1, B_2 \in \mathcal B$ and $C_1, C_2 \in \mathcal C$.

All rings in this paper will be associative with unit (except those in Section 5.3) and
all modules will be right modules. Let $R$ be a ring. As in any
abelian category, we have the abelian exact structure $\mathcal E$ in $\textrm{Mod-}R$ consisting of
all kernel-cokernel pairs. If $\mathcal P$ is the class of all
finitely presented modules, the exact structure $\mathcal E_{\mathcal
  P}$ consists of all pure conflations and will be called the
\textit{pure-exact structure} on $\ModR$. Conflations in the
pure-exact structure can be characterized in terms of systems of
equations \cite[34.5]{Wisbauer}. Given a module $M$, recall that a
\textit{system of linear
  equations over $M$} is a system of equations
\begin{displaymath}
  \sum_{i=1}^n X_ir_{ij} = a_j \quad j \in \{1, \ldots, m\}
\end{displaymath}
with $r_{ij} \in R$ and $a_j\in M$ for each $i \in \{1, \ldots, n\}$
and $j \in \{1, \ldots, m\}$. Then a conflation in $\ModR$,
\begin{displaymath}
  \begin{tikzcd}
    K \arrow{r}{f} & M \arrow{r}{g} & L
  \end{tikzcd}
\end{displaymath}
is pure if and only if any system of linear equations over $\Img f$ that has a
solution over $M$, has a solution over $\Img f$. We shall denote by
$\Inj$ the class of all injective modules and by $\PInj$ the class of
all pure-injective modules (that is, the class of all injective
objects in the exact category $\ModR$ with the pure-exact structure). 


\bigskip

\section{Partial Morphisms}
\label{sec:ginj-peri-modul}

\bigskip

\noindent The initial inspiration for our work comes from the
classical notion of partial morphism introduced by Ziegler in
\cite{Ziegler} in the category of right modules over a ring.

\begin{defn} \label{d:PartialZiegler} Let $R$ be a ring and $M, N$ be
  right $R$-modules.
  \begin{enumerate}
  \item A partial morphism from $M$ to $N$ is a morphism
    $f \colon K \rightarrow N$, where $K$ is a submodule of $M$, such
    that for any system of linear equations over $K$,
    \[\sum_{i=1}^nX_ir_{ij}=k_j \quad j \in \{1, \ldots, m\},\]
    if the system has a solution in $M$, then the system
    \[\sum_{i=1}^nX_ir_{ij}=f(k_j) \quad j \in \{1, \ldots, m\}\] has
    a solution in $N$ as well. We shall call the submodule $K$ the
    domain of $f$ and we shall denote it by $\dom f$.

  \item A partial morphism from $M$ to $N$ is called a partial
    isomorphism if each system of linear equations over $\dom f$,
    \[\sum_{i=1}^nX_ir_{ij}=k_j \quad j \in \{1, \ldots, m\},\]
    has a solution over $M$ if and only if the system of linear equations
    \[\sum_{i=1}^nX_ir_{ij}=f(k_j) \quad j \in \{1, \ldots, m\}\]
    has a solution over $N$.
  \end{enumerate}
\end{defn}

\noindent The following characterization relates partial morphisms
with the pure-exact structure in the categories of modules. It will
allow us to define partial morphisms in any exact category. Let us recall the construction of the pushouts in module categories. Given a ring $R$ and two morphisms $f:K \rightarrow M$ and $g:K \rightarrow N$ in $\textrm{Mod-} R$, the
pushout of $g$ along $f$ is given by the commutative diagram,
\begin{displaymath}
  \begin{tikzcd}
    K \arrow{r}{f} \arrow{d}{g} & M \arrow{d}{\overline g}\\
    N \arrow{r}{\overline f} & P
  \end{tikzcd}
\end{displaymath}
in which the module $P$ can be taken to be $\frac{N \oplus M}{U}$, where
$U=\{(g(k),f(k)):k \in K\}$ and, if we denote by $\overline{(n,m)}$
the corresponding element in $P$ for each $n \in N$ and $m \in M$,
then $\overline{f}(n) = \overline{(n,0)}$ and
$\overline{g}(m) = \overline{(0,-m)}$.

\begin{theorem}\label{p:CharacterizationZieglerPartial}
  Let $R$ be a ring. Let $M$ and $N$ be modules, $K \leq M$ a
  submodule and $f:K \rightarrow N$ a morphism. The following
  assertions are equivalent:
  \begin{enumerate}
  \item $f$ is a partial morphism (resp. isomorphism) from $M$ to $N$
    with $\dom f = K$.

  \item In the pushout diagram
    \begin{displaymath}
      \begin{tikzcd}
        K \arrow[hook]{r}{i} \arrow{d}{f}& M \arrow{d}{\overline f}\\
        N \arrow{r}{\overline i} & P
      \end{tikzcd}
    \end{displaymath}
    $\overline i$ (resp. $\overline i$ and $\overline f$) is a pure
    monomorphism (resp. are pure monomorphisms).
  \end{enumerate}
\end{theorem}

\begin{proof}
  (1) $\Rightarrow$ (2). First assume that $f$ is a partial morphism
  and let us prove that
  $\Img \overline i = \{\overline{(u,0)}:u \in N\}$ is a pure
  submodule of $P$. Let
  \begin{equation}
    \label{eq:1}
    \sum_{i=1}^nX_ir_{ij}=\overline{(s_j,0)} \quad j \in \{1, \ldots, m\}
  \end{equation}
  be a system of linear equations over $\Img \overline i$ which has a
  solution in $P$. Then there exist $u_1, \ldots, u_n \in N$ and
  $v_1, \ldots, v_n\in M$ such that
  $\sum_{i=1}^n \overline{(u_i,v_i)}r_{ij}=\overline{(s_j,0)}$ for
  each $j \in \{1, \ldots, m\}$. Then there exist
  $k_1, \ldots, k_m \in K$ such that
  $\sum_{i=1}^nu_ir_{ij}-s_j = f(k_j)$ and $\sum_{i=1}^nv_ir_{ij}=k_j$
  for each $j \in \{1, \ldots, m\}$. This last equality says that the
  system
  \[\sum_{i=1}^nX_ir_{ij}=k_j \quad j \in \{1, \ldots, m\}\] has a
  solution in $M$ so that, as $f$ is a partial morphism, the system
  \[\sum_{i=1}^nX_ir_{ij}=f(k_j) \quad j \in \{1, \ldots, m\}\] has a
  solution, $u'_1, \ldots, u'_n$, in $N$. Then
  $\overline{(u_1-u'_1,0)}, \ldots, \overline{(u_n-u'_n,0)}$ is a
  solution of (\ref{eq:1}) in $\Img \overline i$. This implies that
  $\Img \overline i$ is a pure submodule of $P$ and
  $\overline i$ is a pure monomorphism.

  Now suppose that $f$ is a partial isomorphism and let us prove that
  $\Img \overline{f} = \{\overline{(0,v)}: v \in M\}$ is a pure
  submodule of $P$. Let
  \begin{equation}
    \label{eq:2}
    \sum_{i=1}^nX_ir_{ij}=\overline{(0,s_j)} \quad j \in \{1, \ldots, m\}
  \end{equation}
  be a system of linear equations over $\Img \overline f$ which has a
  solution in $P$. Then there exist $u_1, \ldots, u_n \in N$ and
  $v_1, \ldots, v_n \in M$ such that
  $\sum_{i=1}^n \overline{(u_i,v_i)}r_{ij}=\overline{(0,s_j)}$ for
  each $j \in \{1, \ldots, m\}$. This implies that there exist
  $k_1, \ldots, k_m \in K$ such that $\sum_{i=1}^nu_ir_{ij} = f(k_j)$
  and $\sum_{i=1}^nv_ir_{ij}-s_j=k_j$ for each
  $j \in \{1, \ldots, m\}$. The first identity says that the system
  \[\sum_{i=1}^nX_ir_{ij}=f(k_j) \quad j \in \{1, \ldots, m\}\]
  has a solution in $N$. Using that $f$ is a partial isomorphism, the
  system \[\sum_{i=1}^nX_ir_{ij}=k_j \quad j \in \{1, \ldots, m\}\]
  has a solution in $M$, say $v'_1, \ldots, v'_n$. Then
  $\overline{(0,v_1-v'_1)}, \ldots, \overline{(0,v_n-v'_n)}$ is a
  solution of (\ref{eq:2}) in $\Img \overline f$. This implies that
  $\Img \overline f$ is a pure submodule of $P$ and
  $\overline f$ is a pure monomorphism.

  (2) $\Rightarrow$ (1). First of all assume that $\overline i$ is a
  pure monomorphism and let
  \[\sum_{i=1}^nX_ir_{ij}=k_j \quad j \in \{1, \ldots, m\}\]
  be a system of linear equations over $K$ which has a solution in $M$. Then
  the system over $\Img \overline i$, 
  \[\sum_{i=1}^nX_ir_{ij}=\overline i f(k_j) \quad j \in \{1, \ldots,
    m\}\] has a solution in $P$ and, using
  that $\overline i$ is pure, it has a solution in $\Img \overline
  i$. Since $\overline i$ is monic, this implies that the system
  \[\sum_{i=1}^nX_ir_{ij}=f(k_j) \quad j \in \{1, \ldots, m\}\]
  has a solution in $N$. Thus, $f$ is a partial morphism.

  Now assume that $\overline f$ is a pure monomorphism too, and let
  \[\sum_{i=1}^nX_ir_{ij}=k_j \quad j \in \{1, \ldots, m\}\]
  be a system of linear equations over $K$ such
  that \[\sum_{i=1}^nX_ir_{ij}=f(k_j) \quad j \in \{1, \ldots, m\}\]
  has a solution in $N$. Then the system
  \[\sum_{i=1}^nX_ir_{ij}=\overline f i( k_j) \quad j \in \{1, \ldots,
    m\}\] has a solution in $P$ and, as $\overline f$ is a pure
  monomorphism, it has a solution in $\Img \overline f$. But, as
  $\overline f$ is monic, this implies that the system
  \[\sum_{i=1}^nX_ir_{ij}=k_j \quad j \in \{1, \ldots, m\}\]
  has a solution in $M$. Thus, $f$ is a partial isomorphism.
\end{proof}

With this characterization we can extend the notion of partial
morphism to any exact category. For the rest of the paper, we fix an
exact category $(\mathcal A;\mathcal E)$ and an additive
exact substructure $\mathcal F$ of $\mathcal E$.

\begin{defn}\label{d:Partial}
  Let $X$ and $Y$ be objects of $\mathcal A$. An $\mathcal F$-partial
  morphism (resp. $\mathcal F$-partial isomorphism) $f$ from $X$ to
  $Y$ is a morphism $f:U \rightarrow Y$, where $U$ is an admissible subobject of $X$ with inclusion $u:U \rightarrow X$, such that in the pushout of $f$
  along $u$,
  \begin{displaymath}
    \begin{tikzcd}
      U \arrow{r}{u} \arrow{d}{f} & X \arrow{d}{\overline{f}}\\
      Y \arrow{r}{\overline u} & P
    \end{tikzcd}
  \end{displaymath}
  $\overline u$ is an $\mathcal F$-inflation (resp. $\overline u$ and
  $\overline f$ are $\mathcal F$-inflations). We shall call the
  subobject $U$ the domain of $f$ and we shall denote it by $\dom f$.
\end{defn}

Sometimes we shall speak about partial morphisms with respect to
$\mathcal F$ instead of $\mathcal F$-partial morphisms. Note that the
definition of $\mathcal F$-partial morphism does not depend on the
selected inclusion $u$ of $U$ since, following the
notation of the definition, if $v:V \rightarrow X$ is an equivalent
monic to $u:U \rightarrow X$ and $w:V \rightarrow U$ is an isomorphism
such that $uw=v$, then $f$ is an $\mathcal F$-partial morphism (resp. isomorphism)
if and only if $fw$ is an $\mathcal F$-partial morphism (resp. isomorphism). We
shall denote by $\dom f$ the subobject $U$ of $X$.

\begin{rem} \rm
In \cite[Definition 28.]{AdamekHerrlichStrecker} another definition of partial morphism is given. For a fixed class $\mathcal M$ of morphisms in a category $\mathcal C$, a $\mathcal M$-partial morphism from $A$ to $B$ is a morphism $f:C \rightarrow B$ defined from an object $C$ for which there exists a morphism $m:C \rightarrow A$ in $\mathcal M$. We would like to emphasize here that this definition has nothing to do with our definition which is inspired by Ziegler partial morphisms.
\end{rem}


Now we obtain some basic properties of partial morphisms:

\begin{prop}\label{p:PropertiesPartialMorphisms}
  Let $X$, $Y$, $Z$ be objects of $\mathcal A$, $U$, an admissible
  subobject of $X$ with inclusion $u:U \rightarrow X$.
  \begin{enumerate}
  \item Suppose that $u$ is an $\mathcal F$-inflation. Then any morphism
    $f:U \rightarrow Y$ is an $\mathcal F$-partial morphism from $X$
    to $Y$ with $\dom f = U$. Moreover, a morphism $f:U \rightarrow Y$
    is an $\mathcal F$-partial isomorphism from $X$ to $Y$ if and only
    if it is an $\mathcal F$-inflation.

  \item If $f:U \rightarrow Y$ is a morphism that has an extension to
    $X$, then $f$ is an $\mathcal F$-partial morphism from $X$ to $Y$
    with $\dom f = U$.

  \item If $f:U \rightarrow Y$ is a morphism, then $f$ defines a
    $\mathcal F$-partial isomorphism from $X$ to $Y$ with $\dom f = U$ if and only
    if $f$ is an inflation, $f$ is an $\mathcal F$-partial morphism
    from $X$ to $Y$ with $\dom f = U$ and $u$ is a
    $\mathcal F$-partial morphism from $Y$ to $X$ with domain the
    subobject $U$ of $Y$ determined by the monomorphism $f$.

  \item Let $f$ be an $\mathcal F$-partial morphism from $X$ to $Y$
    with $\dom f = U$. Then:
    \begin{enumerate}
    \item If there exists $h:Y \rightarrow X$ such that $hf=u$ then
      $f$ is an $\mathcal F$-partial isomorphism.

    \item The converse is true if $X$ is $\mathcal F$-injective.
    \end{enumerate}

  \item Let
    \begin{displaymath}
      \begin{tikzcd}
        \eta: & U \arrow{r}{u} & X \arrow{r}{p} &A
      \end{tikzcd}
    \end{displaymath}
    be a conflation whose kernel is $u$. Then a morphism
    $f:U \rightarrow Y$ defines an $\mathcal F$-partial morphism from
    $X$ to $Y$ with $\dom f=U$ if and only if
    $\Ext(A,f)(\eta) \in \mathcal F$.

  \item If $f$ is an $\mathcal F$-partial morphism (resp.
    $\mathcal F$-partial isomorphism) from $X$ to $Y$ and $g$ is any
    morphism (resp. $\mathcal F$-inflation) from $Y$ to $Z$, then $gf$
    is an $\mathcal F$-partial morphism (resp. $\mathcal F$-partial isomorphism)
    from $X$ to $Z$ with $\dom gf = U$.

  \item If $f$ and $g$ are $\mathcal F$-partial morphisms from $X$ to
    $Y$ with $\dom f = \dom g = U$, then $f+g$ is an $\mathcal F$-partial morphism
    from $X$ to $Y$.

  \item If $f$ is an $\mathcal F$-partial morphism (resp. $\mathcal F$-partial
    isomorphism) from $X$ to $Y$ with $\dom f = U$, and $X$ is an $\mathcal F$-admissible
    subobject of $Z$ with inclusion $v$, then $f$ is an $\mathcal F$-partial morphism
    (resp. $\mathcal F$-partial
    isomorphism) from $Z$ to $Y$ with dominion the subobject $U$ of $Z$ determined by $vu$.
  \end{enumerate}
\end{prop}

\begin{proof}
  (1) The pushout along any $\mathcal F$-inflation is an
  $\mathcal F$-inflation so that any morphism $f:U \rightarrow Y$ is
  $\mathcal F$-partial. Moreover, as a consequence of the obscure axiom
  \cite[Proposition 2.16]{Buhler}, $f$ is an $\mathcal F$-partial isomorphism if and
  only if it is an $\mathcal F$-inflation.

  (2) Let $g:X \rightarrow Y$ be an extension of $f$ and consider the
  pushout of $f$ along $u$:
  \begin{displaymath}
    \begin{tikzcd}
      U \arrow{r}{u} \arrow{d}{f} & X \arrow{d}{f_2}\\
      Y \arrow{r}{f_1} & Q
    \end{tikzcd}
  \end{displaymath}
  Since the identity of $Y$ and $g:X\rightarrow Y$ satisfy $1_Yf=gu$,
  there exists $h:Q \rightarrow Y$ such
  that $hf_1=1_Y$ and $hf_2=g$. Since $f_1$ has a cokernel, as it is an
  inflation, the obscure axiom \cite[Proposition 2.16]{Buhler} says
  that $f_1$ is an $\mathcal F$-inflation. Thus, $f$ is
  $\mathcal F$-partial.

  (3) Note that $f$ is an inflation by the obscure axiom
  \cite[Proposition 2.16]{Buhler} and Lemma
  \ref{l:PushoutCokernel}. The rest of the assertion is trivial.

  (4) Consider the pushout of $f$ and $u$
  \begin{equation*}
    \begin{tikzcd}
      U \arrow{r}{u} \arrow{d}{f} & X \arrow{d}{\overline f}\\
      Y \arrow{r}{\overline u} & P\\
    \end{tikzcd}
  \end{equation*}
  If there exists $h:Y \rightarrow X$ with $hf=u$ then, by Lemma
  \ref{l:PushoutCokernel}, $\overline f$ is a split monomorphism and,
  in particular, an $\mathcal F$-inflation. Thus $f$ is an $\mathcal F$-partial
  isomorphism.

  If $X$ is $\mathcal F$-injective, and $f$ is an $\mathcal F$-partial isomorphism
  then $\overline f$ actually is a split monomorphism. Then there
  exists $h:Y \rightarrow X$ with $hf=u$ by Lemma
  \ref{l:PushoutCokernel}.

  (5) Follows from the definition of $\Ext(A,f)$.

  (6) First assume that $f$ is an $\mathcal F$-partial morphism from $X$ to $Y$ with $\dom f = U$. We get the
  following commutative diagram,
  \begin{equation}
    \label{eq:3}
    \begin{tikzcd}
      \dom f \arrow{r}{u} \arrow{d}{f} & X \arrow{d}{\overline f}\\
      Y \arrow{r}{\overline u} \arrow{d}{g} & P
      \arrow{d}{\overline g}\\
      Z \arrow{r}{\overline v} & Q
    \end{tikzcd}
  \end{equation}
  by considering the pushout of $f$ along $u$ and of $g$ along $\overline u$.
  Then the outer diagram is a pushout and $\overline v$ is an
  $\mathcal F$-inflation, as $f$ is $\mathcal F$-partial. This means that $gf$ is a
  $\mathcal F$-partial morphism from $X$ to $Z$ with $\dom gf = \dom f$.
  
  If, in addition, $g$ is an $\mathcal F$-inflation and $f$ is a
  $\mathcal F$-partial isomorphism from $X$ to $Y$, then in diagram (\ref{eq:3})
  both $\overline f$ and $\overline g$ are $\mathcal F$-inflations, so that,
  $\overline g \overline f$ is an $\mathcal F$-inflation too. Consequently, $gf$ is an
  $\mathcal F$-partial isomorphism from $X$ to $Z$.

  (7) Let
  \begin{displaymath}
    \begin{tikzcd}
      \eta: & U \arrow{r}{u} & X \arrow{r}{p} &A
    \end{tikzcd}
  \end{displaymath}
  be a conflation whose kernel is $u$. Then, since $\Ext(A,-)$ defines
  a morphism of abelian groups,
  $\Ext(A,f+g)(\eta) = \Ext(A,f)(\eta)+\Ext(A,g)(\eta)$. Now using that
  $\mathcal F(A,Y)$ is a subgroup of $\Ext(A,Y)$, we deduce that
  $\Ext(A,f+g)(\eta) \in \mathcal F$. By (5), $f+g$ is an $\mathcal F$-partial
  morphism.

  (8) Let $v:\dom f \rightarrow X$ be an $\mathcal F$-inflation. We can construct the
  following commutative diagram
  \begin{displaymath}
    \begin{tikzcd}
      \dom f \arrow{r}{u} \arrow{d}{f} & X \arrow{r}{v}
      \arrow{d}{\overline f} & Z
      \arrow{d}{\overline g}\\
      Y \arrow{r}{\overline u} & P \arrow{r}{\overline v} & Q
    \end{tikzcd}
  \end{displaymath}
  by considering the pushout of $f$ along $u$ and of $\overline f$ along
  $v$. Then the outer diagram is a pushout and both $\overline u$ and
  $\overline v$ are $\mathcal F$-inflations. Consequently
  $\overline v\circ \overline u$ is an $\mathcal F$-inflation which means
  that $f$ is $\mathcal F$-partial from $Z$ to $Y$ with dominion the subobject $U$ of $Z$ determined by $vu$.

  If, in addition, $f$ is an $\mathcal F$-partial isomorphism, both
  $\overline f$ and $\overline g$ are $\mathcal F$-inflations, then $f$ is a
  $\mathcal F$-partial isomorphism from $Z$ to $Y$.
\end{proof}

\begin{expls}\label{e:PartialMorphisms} \rm
  We give below some examples of partial morphisms and partial
  isomorphisms.

  \medskip

  \begin{enumerate}
  \item Let $X$ and $Y$ be objects in $\mathcal A$ and
    $f:X \rightarrow Y$ be a morphism. Then, by Proposition
    \ref{p:PropertiesPartialMorphisms}(1), $f$ is an
    $\mathcal F$-partial morphism from $X$ to $Y$ with $\dom
    f=X$. Moreover, $f$ is an $\mathcal F$-partial isomorphism with
    $\dom f = X$ if and only if it is an $\mathcal F$-inflation.

    \medskip

  \item Let $R$ be a ring. By Proposition
    \ref{p:CharacterizationZieglerPartial}, the partial morphisms with
    respect to the pure-exact structure in the sense of Definition
    \ref{d:Partial} coincide with those introduced by Ziegler
    (Definition \ref{d:PartialZiegler}).

  \end{enumerate}
\end{expls}

\noindent Phantom morphisms, which have their origin in homotopy theory
\cite{MacGibbon}, were introduced by Gnacadja \cite{Gnacadja} in the
category of modules over a finite group ring, and considered by Herzog
for a general module category in \cite{Herzog}. In
\cite{FuGuilHerzogTorrecillas} phantom morphisms
with respect to the exact substructure $\mathcal F$ have been defined, and also the dual notion of phantom morphisms, the cophantom morphisms have been introduced. A morphism
$f:B \rightarrow Y$ is called $\mathcal F$-cophantom if the pushout of
any conflation (beginning in $B$) along $f$ gives a conflation that belongs to $\mathcal F$
(equivalently, if $\Ext(A,f)(\eta) \in \mathcal F$ for each conflation
of the form $\eta: B \rightarrow C \rightarrow A$). With the preceding
result, it is easy to characterize $\mathcal F$-cophantom morphisms in
terms of $\mathcal F$-partial morphisms.

\begin{cor}\label{c:Cophantom}
  Let $f:B \rightarrow Y$ be a morphism in $\mathcal A$. Then $f$ is
  an $\mathcal F$-cophantom morphism if and only if for any admissible
  inclusion
  $u:B \rightarrow X$, $f$ is $\mathcal F$-partial morphism from $X$
  to $Y$ with $\dom f = B$.
\end{cor}

\noindent In \cite{Ziegler} (see \cite[Theorem 1.1]{Monari} too) Ziegler
characterized pure-injective modules in terms of partial morphisms
with respect to the pure-exact structure. We proceed to extend this
result to injective objects relative to the exact structure
$\mathcal F$.

\begin{theorem}\label{t:FInjectivePartial}
  An object $E$ is $\mathcal F$-injective if and only if any
  $\mathcal F$-partial morphism $f$ from an object $X$ to $E$ extends
  to a morphism $g:X \rightarrow E$.
\end{theorem}

\begin{proof}
  If $E$ is $\mathcal F$-injective and $f$ is an $\mathcal F$-partial
  morphism from an object $X$ to $E$, we can consider the following pushout
  \begin{displaymath}
    \begin{tikzcd}
      \dom f \arrow{r}{v} \arrow{d}{f} & X \arrow{d}{\overline f}\\
      E \arrow{r}{\overline v} & P
    \end{tikzcd}
  \end{displaymath}
  Since $E$ is $\mathcal F$-injective and $\overline v$ is
  an $\mathcal F$-inflation, there exists $w:P \rightarrow E$ with $w \overline v=1_E$. Then
  $w\overline f$ is an extension of $f$ to $X$.

  Conversely, if $v:V \rightarrow X$ is an $\mathcal F$-inflation and
  $f:V \rightarrow E$ is any morphism then, by Proposition
  \ref{p:PropertiesPartialMorphisms}, $f$ is an $\mathcal F$-partial morphism from
  $X$ to $E$. By hypothesis there exists $w:X \rightarrow E$ such that
  $wv=f$. Then $E$ is $\mathcal F$-injective.
\end{proof}

As an application of the preceding theorem we can characterize when a
module belongs to the right-hand class of a cotorsion pair.

\begin{cor}\label{c:CotorsionPair}
  Let $(\mathcal B,\mathcal C)$ be a complete cotorsion pair and $A$, an object of $\mathcal A$. Then the following assertions are
  equivalent:
  \begin{enumerate}
  \item $A \in \mathcal C$.

  \item $A$ is $\mathcal E^{\mathcal C}$-injective.

  \item Any $\mathcal E^{\mathcal C}$-partial morphism from an object
    $X$ to $A$ extends to a homomorphism from $X$ to $A$.
  \end{enumerate}
\end{cor}

\begin{proof}
  (1) $\Rightarrow$ (2) is trivial. (2) $\Leftrightarrow$ (3) follows
  from Theorem \ref{t:FInjectivePartial}.

  (2) $\Rightarrow$ (1). Since the cotorsion pair is complete, there exists a conflation $A \rightarrow B \rightarrow C$ with $C \in \mathcal C$ and
  $B \in \mathcal B$. Then, the long exact sequence induced by this conflation when applying $\Ext(-,C')$ for each $C' \in \mathcal C$, gives that $f$ actually is an
  $\mathcal E^{\mathcal C}$-inflation. Since $A$ is $\mathcal E^{\mathcal C}$-injective, this inflation is a split monomorphism and $A$ is isomorphic to a direct summand of $C$. Now, using that $\mathcal C$ is closed under direct summands, we conclude that $A$ belongs to $\mathcal C$.
\end{proof}

We end this section characterizing partial morphisms relative to the
exact structures $\mathcal E^{\mathcal X}$ and
$\mathcal E_{\mathcal X}$ for a given class of objects $\mathcal X$.
Using the preceding theorem, it is easy to handle the case
$\mathcal E^{\mathcal X}$.

\begin{prop}
  Let $\mathcal X$ be a class of objects, $X$ an object in $\mathcal A$, $U$ an admissible suboject
  with inclusion $u:U \rightarrow X$ and $f:U \rightarrow Y$ be a
  morphism. The following assertions are equivalent:
  \begin{enumerate}
  \item $f$ is an $\mathcal E^{\mathcal X}$-partial morphism from $X$
    to $Y$ with $\dom f = U$.

  \item For each morphism $g:Y \rightarrow Z$ with $Z \in \mathcal X$,
    there exists $h:X \rightarrow Z$ with $hu=gf$.
  \end{enumerate}
\end{prop}

\begin{proof}
  (1) $\Rightarrow$ (2). Take any $Z \in \mathcal X$ and
  $g:Y \rightarrow Z$. By Proposition
  \ref{p:PropertiesPartialMorphisms}(6), $gf$ is a $\mathcal
  E^{\mathcal X}$-partial morphism
  from $X$ to $Z$. Since $Z$ is $\mathcal E^{\mathcal X}$-injective,
  (2) follows from Theorem \ref{t:FInjectivePartial}.

  (2) $\Rightarrow$ (1). Conversely, consider the pushout of $f$ along
  $u$ and a morphism $g:Y \rightarrow Z$ with $Z \in X$:
  \begin{displaymath}
    \begin{tikzcd}
      U \arrow{r}{u} \arrow{d}{f} & X \arrow{d}{\overline f}\\
      Y \arrow{r}{\overline u} \arrow{d}{g} & P\\
      Z &
    \end{tikzcd}
  \end{displaymath}
  By (2) there exists $h:X \rightarrow Z$ such that $hu=gf$. Using
  that $P$ is the pushout, there exists $h':P \rightarrow Z$ such that
  $h'\overline u=g$. This means that $\Hom(P,Z) \rightarrow \Hom(Y,Z)$ is exact and, consequently, $\overline u$ is an
  $\mathcal E^{\mathcal X}$-inflation. Then, $f$ is an
  $\mathcal E^{\mathcal X}$-partial morphism.
\end{proof}

Now we treat the case $\mathcal E_{\mathcal X}$. Having in mind the
interpretation of systems of equations in terms of morphisms (see
\cite[34.3]{Wisbauer}), the following characterization of
$\mathcal E_{\mathcal X}$-partial morphisms can be viewed as an
extension of the definition of Ziegler of partial morphisms in the
pure-exact structure in the module category (Definition
\ref{d:PartialZiegler}).

\begin{prop}
  Suppose that there exist enough projective objects. Let $\mathcal X$
  be a class of objects, $A$ an object, $U$ an admissible subobject of
  $A$ with inclusion $u:U \rightarrow A$ and $f:U \rightarrow B$ a
  morphism. The following assertions are equivalent:
  \begin{enumerate}
  \item $f$ is an $\mathcal E_{\mathcal X}$-partial morphism.
    
  \item For each commutative diagram
    \begin{displaymath}
      \begin{tikzcd}
        M \arrow{r}{i} \arrow{d}{\varphi_1} & N \arrow{d}{\varphi_2}\\
        U \arrow{r}{u} & A
      \end{tikzcd}
    \end{displaymath}
    in which $\Coker i \in \mathcal X$, there exists
    $g:N \rightarrow B$ such that $g i = f \varphi_1$.
  \end{enumerate}
\end{prop}

\begin{proof}
  (1) $\Rightarrow$ (2) Consider a diagram as in (2) and consider the
  pushout of $f$ along $u$ to get the following commutative diagram
  \begin{displaymath}
    \begin{tikzcd}
      & M \arrow{r}{i} \arrow{d}{\varphi_1} & N \arrow{r}{p}
      \arrow{d}{\varphi_2} & \Coker i \arrow{d}{\varphi_3}\\
      & U \arrow{r}{u} \arrow{d}{f} & A \arrow{r}{q} \arrow{d}{\overline
        f} & C \arrow[equal]{d}\\
      \eta: & B \arrow{r}{\overline u} & P \arrow{r}{\overline q} & C
    \end{tikzcd}
  \end{displaymath}
  in which, since $f$ is $\mathcal E_{\mathcal X}$-partial, the bottom
  row is an $\mathcal E_{\mathcal X}$-conflation, $q=\overline{qf}$ is
  a cokernel of $u$ by Lemma \ref{l:PushoutCokernel}, and $\varphi_3$
  exists by the property of the cokernel. Since
  $\Coker i \in \mathcal X$ and $\eta \in \mathcal E_{\mathcal X}$, there exists $h:\Coker i \rightarrow P$
  such that $h\overline q=\varphi_3$. By Lemma \ref{l:DiagramLemma}, there
  exists $g:N \rightarrow B$ such that $gi=f\varphi_1$.

  (2) $\Rightarrow$ (1) The pushout of $f$ along $u$ gives the commutative diagram
  \begin{displaymath}
    \begin{tikzcd}
      & U \arrow{r}{u} \arrow{d}{f} & A \arrow{r}{q}
      \arrow{d}{\overline
        f} & C \arrow[equal]{d}\\
      \eta: & B \arrow{r}{\overline u} & P \arrow{r}{\overline q} & C
    \end{tikzcd}
  \end{displaymath}
  in which $q$ is a cokernel of $u$ by Lemma
  \ref{l:PushoutCokernel}. In order to see that $\eta$ is an
  $\mathcal E_{\mathcal X}$-conflation, let $\varphi:X \rightarrow C$
  be a morphism with $X \in \mathcal X$. Since there exist enough
  projective objects, we can find a conflation
  \begin{displaymath}
    \begin{tikzcd}
      K \arrow{r}{i} & Q \arrow{r}{p} & X
    \end{tikzcd}
  \end{displaymath}
  with $Q$ being projective. Using the projectivity of $Q$, we can construct
  the commutative diagram
  \begin{displaymath}
    \begin{tikzcd}
      K \arrow{r}{i} \arrow{d}{\varphi_1} & Q \arrow{r}{p}
      \arrow{d}{\varphi_2} & X \arrow{d}{\varphi}\\
      U \arrow{r}{u} \arrow{d}{f} & A \arrow{r}{q} \arrow{d}{\overline
        f} & C \arrow[equal]{d}\\
      B \arrow{r}{\overline u} & P \arrow{r}{\overline q} & C
    \end{tikzcd}
  \end{displaymath}
  By hypothesis, there exists $g:Q \rightarrow B$ such that
  $gi=f\varphi_1$. By Lemma \ref{l:DiagramLemma}, there exists
  $h:X \rightarrow P$ such that $\overline q h =\varphi$. Thus, $\eta$
  is an $\mathcal E_{\mathcal X}$-conflation.
\end{proof}

\bigskip

\section{Small Subobjects, Hulls and Envelopes}
\label{sec:small-subobj-envel}

\bigskip

\noindent Approximations by a fixed class of objects are formalized by the notions of
preenvelope and precover. Recall that if $\mathcal B$ is a category,
$\mathcal X$ is a class of objects and $B$ is an object of $\mathcal B$, an
\textit{$\mathcal X$-preenvelope} of $B$ is a morphism
$u:B \rightarrow X$, with $X$ being an object in $\mathcal X$, such that any
morphism $f:B \rightarrow Y$ with $Y \in \mathcal X$ factors through
$u$. Note that if $\mathcal B$ is the module category over a ring $R$,
then an
$\Inj$-preenvelope is just a monomorphism $B \rightarrow I$ with
$I$ injective and a $\PInj$-preenvelope is a pure monomorphism $B
\rightarrow E$ with $E$ pure-injective.

There are two ways of defining a minimal approximation in module
categories. The first of
them, which can be defined in any category, is the notion of envelope: an $\mathcal X$-preenvelope $u:B
\rightarrow X$ is an $\mathcal X$-envelope if $u$ is a minimal
morphism in the sense that any morphism $f:X \rightarrow X$
satisfying $fu=u$ is an isomorphism.

The second of them uses the notion of essential and pure-essential
monomorphism. Recall that a monomorphism (resp. a pure monomorphism) $f:A \rightarrow B$ is
essential (resp. pure-essential) if for any $g:B \rightarrow C$ such
that $gf$ is a monomorphism (resp. a pure monomorphism), then $g$ is a
monomorphism (resp. pure monomorphism). Then an injective hull in
$\ModR$ is an essential monomorphism $u:B \rightarrow I$ with $I$
injective, and a pure-injective hull is a pure-essential pure
monomorphism $v:B \rightarrow E$ with $E$ pure-injective (we shall use
the term \textit{hull} for minimal approximations defined by
essentiality). It is well known that $u$ is precisely the injective
envelope of $B$ and $v$ the pure-injective envelope of $v$ (as defined
in the preceding paragraph).

Concerning the pure-exact structure, there is another notion of small
extension which was introduced by Ziegler in
\cite[p. 161]{Ziegler} using partial morphisms. With this definition Ziegler constructs, for a
submodule $A$ of a pure-injective module $E$, a weak version of the
pure-injective hull of $A$, $A \leq H(A) \leq E$ (see \cite[Theorem
3.6]{Ziegler}) which gives, in case $A$ is a pure submodule of $E$,
the pure-injective hull of $A$.

The objective of this section is to define $\mathcal F$-essential and
$\mathcal F$-small extensions in our exact category
$(\mathcal A;\mathcal E)$, and to relate all approximations of objects by injectives:
$\mathcal F$-injective envelopes, $\mathcal F$-injective hulls and
$\mathcal F$-small extensions.

We shall start with the definition of $\mathcal F$-small
extension. Note that if $X$ and $Y$ are objects in $\mathcal A$, $f$
is an $\mathcal F$-partial morphism from $X$ to $Y$ and $V$ is an admissible
subobject of $\dom f$, then $f\upharpoonright V$ defines an $\mathcal F$-partial
morphism from $X$ to $Y$ (with $\dom f\upharpoonright V = V$).

\begin{defn}\label{d:small}
  Let $X$ be an object and $U \leq V$ be admissible subobjects of $X$.
  \begin{enumerate}
  \item We shall say that $V$ is $\mathcal F$-small over $U$ in $X$ if
    for any $\mathcal F$-partial morphism $f$ from $X$ to another
    object $Y$ with $\dom f = V$, the following holds:
    \begin{center}
      $f \upharpoonright U$ is an $\mathcal F$-partial isomorphism
      from $X$ to $Y \Rightarrow f$ is an $\mathcal F$-partial
      isomorphism.
    \end{center}
  \item We shall say that $X$ is $\mathcal F$-small over $U$ if $X$ is
    $\mathcal F$-small over $U$ in $X$.
  \end{enumerate}
\end{defn}

\noindent If $R$ is a ring, $\mathcal A=\ModR$ and $\mathcal F$ is the
pure-exact structure in $\ModR$, then the $\mathcal F$-small objects
coincide with the small objects introduced by Ziegler in
\cite{Ziegler}.

\noindent As an immediate consequence of the above definition we get:

\begin{lem}\label{l:CharSmall}
  Let $X$ be an object and $U \leq X$ be an admissible subobject. Then
  $X$ is $\mathcal F$-small over $U$ if and only if each morphism
  $f:X \rightarrow Z$ such that $f \upharpoonright U$ defines an
  $\mathcal F$-partial isomorphism from $X$ to $Z$ is actually an
  $\mathcal F$-inflation.
\end{lem}

\begin{proof}
  Simply note that, by Example \ref{e:PartialMorphisms}, any morphism
  $f:X \rightarrow Z$ is $\mathcal F$-partial with $\dom f = X$ and that $f$ is a
  $\mathcal F$-partial isomorphism with $\dom f = X$ if and only if $f$ is an
  $\mathcal F$-inflation.
\end{proof}

Next, we establish some fundamental properties of $\mathcal F$-small
objects.

\begin{prop}\label{p:PropertiesSmall}
  Let $X$ be an object and $U \subseteq V \subseteq W$ be admissible subobjects
  of $X$. Then:
  \begin{enumerate}
  \item If $V$ is $\mathcal F$-small over $U$ in $X$ and $W$ is
    $\mathcal F$-small over $V$ in $X$ then $W$ is $\mathcal F$-small
    over $U$ in $X$.

    \medskip

  \item If $X$ is $\mathcal F$-injective then $V$ is
    $\mathcal F$-small over $U$ in $X$ if and only if for each $\mathcal F$-partial
    morphism from $X$ to $Y$ with $\dom f = V$ we have that: if the inclusion $u:U \rightarrow X$ factors through
    $f \upharpoonright U$, then the inclusion $v:V \rightarrow X$
    factors through $f$.

    \medskip

  \item If $V$ is an $\mathcal F$-admissible subobject of $X$, then
    $V$ is $\mathcal F$-small over $U$ in $X$ if and only if $V$ is
    $\mathcal F$-small over $U$ (in $V$).
  \end{enumerate}
\end{prop}

\begin{proof}
  (1) is straightforward. (2) follows from the description of $\mathcal F$-partial isomorphisms defined over $\mathcal F$-injective objects obtained in Proposition \ref{p:PropertiesPartialMorphisms}(4).
  
  \medskip

  (3) First of all assume that $V$ is $\mathcal F$-small over $U$ in
  $X$ and let us use the preceding lemma to prove that $V$ is
  $\mathcal F$-small over $U$. Take any morphism $f:V \rightarrow Y$
  such that $f \upharpoonright U$ is an $\mathcal F$-partial
  isomorphism. Since $V$ is an $\mathcal F$-admissible subobject, $f$
  is an $\mathcal F$-partial morphism from $X$ to $Y$ with
  $\dom f = V$ by Proposition \ref{p:PropertiesPartialMorphisms}(1).
  Since $V$ is small over $U$ in $X$, $f$ is an $\mathcal F$-partial
  isomorphism from $X$ to $Y$ with dominion $V$. Again by Proposition
  \ref{p:PropertiesPartialMorphisms}(1), $f$ is an
  $\mathcal F$-inflation.

  Now assume that $V$ is $\mathcal F$-small over $U$ and let $f$ be an
  $\mathcal F$-partial morphism from $X$ to an object $Y$ with $\dom f = V$ such
  that $f \upharpoonright U$ defines an $\mathcal F$-partial isomorphism from $X$ to
  $Y$. Then, trivially, $f \upharpoonright U$ defines an $\mathcal F$-partial
  isomorphism from $V$ to $Y$ and, since $V$ is $\mathcal F$-small over $U$, $f$ is
  an $\mathcal F$-inflation by Lemma \ref{l:CharSmall}. Since $V$ is
  $\mathcal F$-admissible, $f$ is an $\mathcal F$-partial isomorphism from $X$ to
  $Y$ by Proposition \ref{p:PropertiesPartialMorphisms}(1).
\end{proof}

With the notion of $\mathcal F$-small objects we can define $\mathcal
F$-small extensions.

\begin{defn}
  An $\mathcal F$-small extension is an inflation $f:U \rightarrow X$
  such that $X$ is $\mathcal F$-small over $U$.
\end{defn}

The following characterization follows from the definition of partial
isomorphism with respect to the pure-exact structure.

\begin{prop}
  Let $R$ be a ring. A monomorphism $v:U \rightarrow X$ is a
  pure-small extension if and only if any morphism $g:X \rightarrow Y$
  is a pure monomorphism provided that it satisfies the following:
  \begin{enumerate}
  \item $gu$ is monic.

  \item For each system of linear equations over $U$,
    $\sum_{j =1}^m X_jr_{ij}=u_i \ (i=1, \ldots, n)$, if
    $\sum_{j=1}^mX_jr_{ij}=gv(u_i) \ (i=1, \ldots, n)$ has a solution
    in $Y$, then $\sum_{j =1}^m X_jr_{ij}=u_i\ (i=1, \ldots, n)$ has a
    solution in $X$.
  \end{enumerate}
\end{prop}

\begin{rem} \rm
  Note that $g:X \rightarrow Y$ is a pure monomorphism if and only if:
  \begin{enumerate}
  \item $g$ is monic.

  \item Each system of linear equations over $X$,
    $\sum_{j =1}^m X_jr_{ij}=x_i\ (i=1, \ldots, n)$, satisfies that if
    the system $\sum_{j =1}^m X_jr_{ij}=g(x_i)\ (i=1, \ldots, n)$ has
    a solution in $Y$, then the system
    $\sum_{j =1}^m X_jr_{ij}=x_i\ (i=1, \ldots, n)$ has a solution in
    $X$.
  \end{enumerate}
  The previous result says that, when $X$ has a submodule $U$ such
  that the extension $U \leq X$ is pure-small, then we only have to
  check the condition on systems of equations over $U$ in order to see
  that a morphism $g:X \rightarrow Y$ is a pure monomorphism.
\end{rem}

Now we define $\mathcal F$-essential extensions and weakly $\mathcal
F$-essential extensions.

\begin{defn}
  A weakly $\mathcal F$-essential extension (resp.
  $\mathcal F$-essential extension) is an $\mathcal F$-inflation
  $u\colon X \rightarrow Y$ such that for any morphism
  $f:Y \rightarrow Z$, the following holds:
  \begin{center}
    $f u$ is an $\mathcal F$-inflation $\Rightarrow$ $f$ is an
    inflation (resp. $f$ is an $\mathcal F$-inflation).
  \end{center}
\end{defn}

\noindent If $\mathcal A = \ModR$ and $\mathcal E$ is the abelian exact
structure, then both
the weakly $\mathcal E$-essential extensions and the $\mathcal E$-essential
extensions coincide, since each monic is an inflation. If we consider
$\mathcal F$ to be the pure-exact structure on $\ModR$, then the weakly
$\mathcal F$-essential
extensions are the pure-essential
extensions introduced in \cite{Warfield}; we shall call them weakly
pure-essential. The $\mathcal F$-essential extensions are the purely essential
monomorphisms introduced in \cite{GomezGuil} (caution: they are called
pure-essential in \cite[p. 45]{Prest09}). We shall use the name pure-essential extension. In \cite[Example 2.3]{GomezGuil} it is
proved that there exist weakly pure-essential extensions which are not
pure-essential.

We establish the relationship between $\mathcal F$-essential
extensions and $\mathcal F$-small extensions in the sense of
Definition \ref{d:small}.

\begin{prop}\label{p:EssentialSmall}
  Let $u\colon X \rightarrow Y$ be an inflation.
  \begin{enumerate}
  \item The following assertions are equivalent:
    \begin{enumerate}
    \item $u$ is an $\mathcal F$-essential extension.

    \item $u$ is an $\mathcal F$-inflation and $Y$ is
      $\mathcal F$-small over $X$.
    \end{enumerate}
  \item If $u$ is a weakly $\mathcal F$-essential extension then $u$ does
    not factor through a proper direct summand of $Y$, that is, if
    $v:Z \rightarrow Y$ is a split monomorphism and
    $w:X \rightarrow Z$ is an inflation such that $vw=u$, then $v$ is
    an isomorphism.
  \end{enumerate}
\end{prop}

\begin{proof}
  (1) First of all, suppose that $u$ is an
  $\mathcal F$-essential extension and let us prove that $Y$ is small
  over $X$. We will use Lemma \ref{l:CharSmall}. Let
  $f:Y \rightarrow Z$ be a morphism such that
  $f\upharpoonright X = fu$ defines an $\mathcal F$-partial
  isomorphism from $Y$ to $Z$. Since $X$ is an $\mathcal F$-admissible
  subobject, $f\upharpoonright X$ is actually an
  $\mathcal F$-inflation by Proposition
  \ref{p:PropertiesPartialMorphisms}(1). Since $u$ is an
  $\mathcal F$-essential extension, $f$ is an $\mathcal F$-inflation.
  By Lemma \ref{l:CharSmall}, $Y$ is $\mathcal F$-small over $X$.
  
  Conversely, assume that $u$ is an $\mathcal F$-inflation and $Y$ is
  $\mathcal F$-small over $X$. Let $f:Y \rightarrow Z$ be a morphism
  such that $f \upharpoonright X=fu$ is an $\mathcal
  F$-inflation. Then, by Proposition
  \ref{p:PropertiesPartialMorphisms}(1), $f\upharpoonright X$ defines
  an $\mathcal F$-partial isomorphism from $Y$ to $Z$. Since $Y$ is
  $\mathcal F$-small over $X$, $f$ is an $\mathcal F$-inflation by
  Lemma \ref{l:CharSmall}. Thus, $u$ is an
  $\mathcal F$-essential extension.

  (2) Let $v:Z \rightarrow Y$ be a split monomorphism, $v':Y \rightarrow Z$, a morphism with $v'v=1_Z$ and $w:X \rightarrow Z$ with an inflation with $vw=u$. Since $w$ is
  an inflation, $w$ is an $\mathcal F$-inflation by the obscure
  axiom. Using that $v'u=w$ and that $u$ is weakly $\mathcal F$-essential, we get that $v'$
  is monic. Then $v'vv'=v'=v'1_Y$ from which it follows that $vv'=1_Y$
  and, consequently, $v$ is an isomorphism.
\end{proof}

With the notion of $\mathcal F$-essential extension we can define
$\mathcal F$-injective hulls.

\begin{defn}
  An $\mathcal F$-injective hull of an object $X$ is an
  $\mathcal F$-essential extension $u:X \rightarrow E$ with $E$, an
  $\mathcal F$-injective object.
\end{defn}

In the next result we see that, under certain circumstances, a
weakly $\mathcal F$-essential extension $u:X \rightarrow E$ with $E$ being
$\mathcal F$-injective is actually an $\mathcal F$-injective hull. In
addition, we establish the relationship between $\mathcal F$-injective
hulls and $\mathcal F$-injective envelopes as defined at the beginning
of this section. If $\FInj$ is the class of all
$\mathcal F$-injective objects, we shall call $\FInj$-envelopes to be
$\mathcal F$-injective envelopes.

\begin{theorem}\label{t:InjectiveHulls}
  Let $u:X \rightarrow Y$ be a morphism. The following assertions are
  equivalent:
  \begin{enumerate}
  \item $u$ is an $\mathcal F$-injective hull.

  \item $u$ is an $\mathcal F$-inflation, $Y$ is
    $\mathcal F$-injective and $Y$ is $\mathcal F$-small over $X$.

  \item $u$ is an $\mathcal F$-inflation, $Y$ is
    $\mathcal F$-injective and each morphism $f:Y \rightarrow Z$
    satisfying that $fu$ is an $\mathcal F$-inflation, is a split
    monomorphism.
  \end{enumerate}
  If, in addition, $u$ has a cokernel and there exists an
  $\mathcal F$-inflation $v:X \rightarrow E$ with $E$ being a
  $\mathcal F$-injective object, the following assertion is equivalent
  too:
  \begin{enumerate}
    \setcounter{enumi}{3}
  \item $u$ is an $\mathcal F$-injective envelope.
  \end{enumerate}
  Finally, if there exists an $\mathcal F$-essential extension
  $v:X \rightarrow E$ with $E$ an $\mathcal F$-injective object, the
  following assertion is equivalent too:
  \begin{enumerate}
    \setcounter{enumi}{4}
  \item $u$ is a weakly $\mathcal F$-essential extension with $Y$
    being $\mathcal F$-injective.
  \end{enumerate}
\end{theorem}

\begin{proof}
  (1) $\Leftrightarrow$ (2) is Proposition \ref{p:EssentialSmall} and
  (1) $\Leftrightarrow$ (3) is trivial.

  (1) $\Rightarrow$ (4). Since $u$ is an $\mathcal F$-inflation, it is
  an $\mathcal F$-injective preenvelope. In order to see that it is an
  envelope let $f:Y \rightarrow Y$ be a morphism such that
  $fu=u$. Since $u$ is $\mathcal F$-essential, $f$ is an
  $\mathcal F$-inflation. Using that $Y$ is $\mathcal F$-injective, we
  deduce that $f$ is a splitting monomorphism, i. e., there exists
  $g:Y \rightarrow Y$ such that $gf=1_Y$. Then $gu=gfu=u$ and, in
  particular, $g$ is a monomorphism. Then $gfg=g=g1_Y$. In particular,
  $fg=1_Y$, which implies that $f$ is an isomorphism.

  (4) $\Rightarrow$ (3). Since $u$ is an $\mathcal F$-injective
  envelope, there exist $w:Y \rightarrow E$ such that $wu=v$. By the
  obscure axiom \cite[Proposition 2.16]{Buhler}, $u$ is an
  $\mathcal F$-inflation. Now let $f:Y \rightarrow Z$ be a morphism
  such that $fu$ is an $\mathcal F$-inflation. Since $Y$ is
  $\mathcal F$-injective, there exists $g:Z \rightarrow Y$ such that
  $gfu=u$. Using that $u$ is an $\mathcal F$-injective envelope we get
  that $gf$ is an isomorphism. This implies that $f$ is a split monic.

  (5) $\Rightarrow$ (1). Since $v$ is an $\mathcal F$-inflation and $Y$
  is $\mathcal F$-injective, there exists $w:E \rightarrow Y$ with
  $wv=u$. Since $u$ is $\mathcal F$-inflation and $v$ is
  $\mathcal F$-essential, $w$ is an $\mathcal F$-inflation. Using that $E$
  is $\mathcal F$-injective, there exists
  $w':Y \rightarrow E$ such that $w'w=1_E$. Then $w'u=v$ is an
  $\mathcal F$-inflation so that, since $u$ is weakly $\mathcal F$-essential,
  $w'$ has to be monic. Then $w'ww'=w'1_Y$ implies that $ww'=1_Y$, so
  that $w$ is an isomorphism. Now the identity $wv=u$ gives that $u$ is
  $\mathcal F$-essential as well.
\end{proof}



\begin{rem}
  Note that the additional hypotheses of (4) (resp. (5)) are
    only needed to prove the implication $(4) \Rightarrow (1)$ (resp.
    $(5) \Rightarrow (1)$). The implication $(1) \Rightarrow (4)$
    (resp. $(1) \Rightarrow (5)$) is true without those hypotheses. In
    particular, any $\mathcal F$-injective hull is always an
    $\mathcal F$-injective envelope.
\end{rem}

\noindent Let $R$ be any ring. In \cite[Proposition 6]{Warfield} it is
proved that for each module $M$ there exists a weakly pure-essential
extension $u:M \rightarrow E$ with $E$ a pure-injective module. In view of the
preceding result, $u$ need not be the pure-injective hull of
$M$. However, one can prove that pure-injective hulls exist by using
the existence of injective hulls in the functor category \cite[Theorem
4.3.18]{Prest09}, so that, by (5) of the preceding theorem, $u$ is
actually pure-essential. That is, \cite[Proposition
6]{Warfield} actually gives the existence of pure-injective hulls in
$\ModR$.


\bigskip

\section{Existence of hulls and envelopes}
\label{sec:exist-hulls-envel}

\bigskip

\noindent In this section we study the problem of existence of
injective hulls and envelopes in our exact category $\mathcal
A$. First, we study when there does exist enough injectives
(equivalently, injective preenvelopes). Then, we prove that in certain
abelian categories this preenvelopes can be used to produce injective
envelopes and hulls.

\noindent Recall that a $\lambda$-sequence, where $\lambda$ is an
ordinal, is a direct system of objects of $\mathcal A$,
$(X_\alpha,i_{\beta\alpha})_{\alpha<\beta<\lambda}$, which is continuous in the sense that for
each limit ordinal $\beta$, the direct limit of the system
$(X_\alpha,i_{\gamma\alpha})_{\alpha<\gamma<\beta}$ exists and the
canonical morphism
$\displaystyle \lim_{\substack{\longrightarrow\\ \alpha <
    \beta}}X_\alpha \rightarrow X_\beta$ is an isomorphism. If the
direct limit of the system exists, we shall call the morphism
$\displaystyle X_0 \rightarrow \lim_{\longrightarrow}X_\alpha$ the
transfinite composition of the $\lambda$-sequence. In many results of
this section we shall use that transfinite compositions of inflations
exist and are inflations. When this condition is satisfied, the
category $\mathcal A$ has arbitrary direct sums and direct sums of
conflations are conflations \cite[Lemma
1.4]{SaorinStovicek11}. Moreover, it is easy to see that when direct
limits of inflations are inflations, then transfinite compositions of
$\lambda$-sequences of inflations are inflations for each ordinal number $\lambda$.

Now we define the notion of small object. Given an object $X$, and a
direct system in $\mathcal A$, $(Y_i,u_{ji})_{i < j \in I}$, such that
its direct limit exists, the functor $\Hom_{\mathcal A}(X,-)$ is said
to \textit{preserve the direct limit of the system} if the canonical
morphism from
$\displaystyle \lim_{\longrightarrow}\Hom_{\mathcal A}(X,Y_i)$ to
$\displaystyle \Hom_{\mathcal
  A}\left(X,\lim_{\longrightarrow}Y_i\right)$ is an isomorphism. It is
very easy to see the following \cite[p. 9]{AdamekRosicky}:

\begin{lem}\label{l:PreserveLimits}
  Let $X$ be an object and $(Y_i,u_{ji})_{i < j \in I}$ a direct
  system such that its direct limit exists, and denote by
  $\displaystyle u_i:Y_i \rightarrow \lim_{\longrightarrow}Y_j$ the
  canonical map for each $i \in I$. Then $\Hom_{\mathcal A}(X,-)$
  preserves the direct limit of the system if and only if the
  following conditions hold:
  \begin{enumerate}
  \item For each
    $\displaystyle f:X \rightarrow \lim_{\longrightarrow}Y_j$ there
    exists $i \in I$ and $g:X \rightarrow Y_i$ such that $f=u_ig$.

  \item For each $i \in I$ and morphism $g:X \rightarrow Y_i$
    satisfying $u_ig=0$, there exists $j \geq i$ such that
    $u_{ji}g=0$.
  \end{enumerate}
\end{lem}

Recall that the cofinality of a cardinal $\kappa$ is the least
cardinal, denoted $\cf(\kappa)$, such that there exists a family of
smaller cardinals than $\kappa$,
$\{\kappa_\alpha:\alpha < \cf(\kappa)\}$, whose union is $\kappa$. The
cardinal $\kappa$ is said to be regular if $\cf(\kappa)=\kappa$.

\begin{defn}
  Suppose that transfinite compositions of inflations exist and are inflations. Let $\kappa$ be an infinite regular cardinal and $X$ be an
  object. We say that $X$ is $\kappa$-small if for each cardinal
  $\lambda$ with $\cf(\lambda) \geq \kappa$, $\Hom_{\mathcal A}(X,-)$
  preserves the transfinite composition of any $\lambda$-sequence of
  inflations. We say that the object $X$ is small if it is
  $\kappa$-small for some infinite regular cardinal $\kappa$.
\end{defn}

  \begin{lem}
    Suppose that transfinite compositions of inflations exist and are inflations. Let $\kappa$ be an infinite regular cardinal and $\{X_k:k \in K\}$
    a family of $\kappa$-small objects with $|K| < \kappa$. Then
    $\bigoplus_{k \in K}X_k$ is $\kappa$-small. In particular, the
    direct sum of any family of small objects is small.
  \end{lem}

  \begin{proof}

    Let $\lambda$ be any cardinal with $\cf(\lambda) \geq \kappa$ and
    $(Y_\alpha,u_{\beta\alpha})_{\alpha < \beta < \lambda}$, a
    $\lambda$-sequence of inflations whose direct limit is $Y$. Denote by $u_\alpha:Y_\alpha \rightarrow Y$ the canonical morphism for each
    $\alpha < \lambda$. We are going to use Lemma
    \ref{l:PreserveLimits} in order to prove that $\bigoplus_{k \in K}X_k$ is $\kappa$-small. Let
    $f:\bigoplus_{k \in K}X_k \rightarrow Y$ be a morphism and denote
    by $\tau_k:X_k \rightarrow \bigoplus_{k \in K}X_k$ the inclusion
    for each $k \in K$. Since, for each $k \in K$, $X_k$ is
    $\kappa$-small, there exists $\alpha_k < \lambda$ and a morphism
    $g_k:X_k \rightarrow Y_{\alpha_k}$ such that
    $u_{\alpha_k}g_k=f\tau_k$. Since $|K| < \cf(\lambda)$, we can find
    an ordinal $\alpha$ with $\alpha_k < \alpha$ for each $k \in
    K$. Now let $g:\bigoplus_{k \in K}X_k \rightarrow Y_\alpha$ be the
    morphism induced in the direct sum by the family
    $\{u_{\alpha\alpha_k}g_k:k \in K\}$ and note that $g$ satisfies
    $u_\alpha g = f$, as $u_\alpha g \tau_k = f\tau_k$
    for each $k \in K$. This proves (1) of Lemma
    \ref{l:PreserveLimits}.

    In order to prove (2), let $\alpha < \lambda$ and
    $f:\bigoplus_{k \in K}X_k \rightarrow Y_\alpha$ such that
    $u_\alpha f=0$. Since, for each $k \in K$, $X_k$ is
    $\kappa$-small, there exists $\alpha_k \geq \alpha$ such that
    $u_{\alpha_k\alpha}f\tau_k=0$. Using that $|K|<\cf(\lambda)$,
    there exists a $\beta < \lambda$ such that $\alpha_k < \beta$ for
    each $k \in K$. Then $u_{\beta\alpha}f\tau_k=0$ for each
    $k \in K$. This means that $u_{\beta\alpha}f=0$, which proves (2)
    of Lemma \ref{l:PreserveLimits}.
  \end{proof}

  Now we can prove the existence of enough injective objects in
  exact categories satisfying that transfinite compositions of inflations exist and are inflations, and a certain generalized version of Baer's lemma for injectivity.

\begin{theorem}\label{t:ExistenceInjectives}
  Assume that the exact category $(\mathcal A;\mathcal E)$ satisfy the
  following:
  \begin{enumerate}
  \item Transfinite compositions of inflations exist and are
    inflations.

  \item There exists a set of inflations
    $\mathcal H = \{u_i:K_i \rightarrow H_i|i \in I\}$ such that $K_i$
    is small for each $i \in I$ and any $\mathcal H$-injective object
    is injective.
  \end{enumerate}
  Then $\mathcal A$ has enough injectives.
\end{theorem}

\begin{proof}
  Let $M$ be any object of $\mathcal A$. Let $J$ be the set of all
  pairs $(i,f)$, where $i$ is an element of $I$ and
  $f:K_i \rightarrow M$ is a morphism. For any pair $(i,f) \in J$, let 
  $u_{(i,f)}:K_{(i,f)}\rightarrow H_{(i,f)}$ be a copy of $u_i$,
  where $K_{(i,f)} = K_i$ and $H_{(i,f)} = H_i$, and compute $u$ the
  induced morphism from
  $\displaystyle \bigoplus_{(i,f) \in J}K_{(i,f)}$ to
  $\displaystyle \bigoplus_{(i,f) \in J}H_{(i,f)}$ by all these
  inclusions. By the properties of $\mathcal H$, it is easy to see
  that an object $E$ is injective if and only if it is
  $u$-injective. Denote
  $\displaystyle \bigoplus_{(i,f) \in J}K_{(i,f)}$ by $K$ and
  $\displaystyle \bigoplus_{(i,f) \in J}H_{(i,f)}$ by $H$. Since $K_i$
  is small for each $i \in I$, we can apply Lemma
  \ref{l:PreserveLimits} to find an infinite regular cardinal $\kappa$
  such that $K$ is $\kappa$-small.

  Now we are going to construct a family of objects
  $\{P_\alpha:\alpha < \kappa\}$ and of inflations
  $\{f_{\alpha\beta} \in \Hom(P_\beta,P_\alpha): \alpha \leq \beta <
  \kappa\}$ such that:
  \begin{enumerate}
  \item[(A)] $P_0=M$.

  \item[(B)] For each $\alpha < \kappa$, the system
    $(P_\gamma,f_{\delta\gamma})_{\gamma < \delta \leq \alpha}$ is
    direct.

  \item[(C)] For each $\alpha < \kappa$ and
    $f:K \rightarrow P_\alpha$, there exists
    $g:H \rightarrow P_{\alpha+1}$ with $gu=f_{\alpha+1,\alpha}f$.
  \end{enumerate}

  We make the construction by transfinite recursion. Suppose that $\alpha$ is a limit ordinal and that we have made the construction for all
 $\gamma < \alpha$. Then set $P_\alpha = \varinjlim_{\gamma < \alpha}P_\gamma$ and, for each $\gamma < \alpha$, set $f_{\alpha\gamma}$ the canonical morphism associated to this direct limit.
  
  Now suppose that we have made the construction for the ordinal $\alpha$ and let us make it for $\alpha+1$. For each morphism $f \in \Hom(K,P_\alpha)$, let
  $K^\alpha_f$ and $H^\alpha_f$ be a copies of $K$ and $H$
  respectively. Denote by $I_\alpha = \Hom(K,P_\alpha)$, let
  $u_\alpha:\bigoplus_{f \in I_\alpha}K^\alpha_f \rightarrow \bigoplus_{f \in
    I_\alpha}H^\alpha_f$ be the direct sum of copies of $u$, and
  $\varphi_\alpha:\bigoplus_{f \in I_\alpha}K^\alpha_f \rightarrow
  P_{\alpha}$ the morphism induced in the direct sum by all morphism
  from $K$ to $P_\alpha$. Then take $P_{\alpha+1}$ and
  $f_{\alpha+1,\alpha}$ the lower arrow in the pushout of $u_\alpha$
  along $\varphi_\alpha$:
  \begin{displaymath}
    \begin{tikzcd}
      \bigoplus_{f \in I_\alpha}K^\alpha_f \arrow{r}{u_\alpha}
      \arrow{d}{\varphi_\alpha} & \bigoplus_{f \in I_\alpha}H^\alpha_f
      \arrow{d}{\psi_\alpha}\\P_\alpha \arrow{r}{f_{\alpha+1,\alpha}}
      & P_{\alpha+1}
    \end{tikzcd}
  \end{displaymath}
  Moreover, set $f_{\alpha+1,\gamma} = f_{\alpha+1,\alpha}f_{\alpha,\gamma}$ for each $\gamma < \alpha$.
  Let us prove that $P_{\alpha+1}$ and $f_{\alpha+1,\alpha}$ satisfy
  (C). Let us denote, for each $f \in I_\alpha$, by $i_f$ and $k_f$ the
  corresponding inclusions of $K_f^{\alpha}$ and $H_f^{\alpha}$ in
  $\bigoplus_{f \in I_\alpha}K^\alpha_f$ and $\bigoplus_{f \in I_\alpha}H^\alpha_f$
  respectively. Given
  $f:K \rightarrow P_{\alpha}$, note that $u_\alpha i_f = u k_f$ and $f=\varphi_\alpha i_f$. Consequently:
  \begin{displaymath}
    f_{\alpha+1,\alpha}f = \psi_\alpha u_\alpha i_f = \psi_\alpha k_f u
  \end{displaymath}
  Then the morphism $g=\psi_\alpha k_f$ satisfy (C). This concludes
  the construction.

  Finally, let
  $\displaystyle E = \lim_{\substack{\longrightarrow\\ \alpha <
      \kappa}}P_\alpha$ and denote by
  $f_\alpha:P_\alpha \rightarrow E$ the canonical maps associated to
  this direct limit. By (1),
  $f_0:M \rightarrow E$ is an inflation. Let us prove that $E$ is injective which, by (2), is equivalent to see that $E$ is
  $u$-injective. Let
  $f:K \rightarrow E$ be any morphism. Since $K$ is $\kappa$-small,
  there exists, by Lemma \ref{l:PreserveLimits}, an $\alpha < \kappa$
  and a morphism $\overline f:K \rightarrow P_\alpha$ such that
  $f = f_\alpha \overline f$. By the construction of $E$, there exists
  $\overline g:F \rightarrow P_{\alpha+1}$ such that
  $f_{\alpha+1,\alpha}\overline f = \overline g u$. Then
  $f = f_{\alpha+1}\overline g u$, and the proof is finished.
\end{proof}

\begin{rem}\label{r:StructurePreenvelope}
Let $M$ be an object of $\mathcal A$. Note that the $\mathcal F$-inflation $i:M \rightarrow E$ with $E$ an $\mathcal F$-injective object constructed in the preceding proof satisfies the following property: there exists an infinite regular cardinal $\kappa$ and a $\kappa$-sequence $(P_\alpha,f_{\beta\alpha})_{\alpha < \beta < \kappa}$ such that $P_0=M$, $i$ is the transfinite composition of the sequence and, for each $\alpha < \kappa$, $f_{\alpha+1,\alpha}$ is a pushout of a direct sum of inflations belonging to $\mathcal H$.
\end{rem}

\begin{rem}\label{r:Hyphoteses}
  Note that (2) in the preceding theorem is satisfied for those exact
  categories for which there exists a set of objects $\mathcal G$ such
  that:
  \begin{enumerate}
  \item The class of admissible subobjects of any $G \in \mathcal G$
    is a set.

  \item Any admissible subobject of any $G \in \mathcal G$ is small.

  \item If an object $A$ of $\mathcal A$ is $\mathcal G$-injective,
    then it is injective.
  \end{enumerate}
  In this case, we only have to take $\mathcal H$ as the set of all
  inflations $u:K \rightarrow G$ with $G$ an object in $\mathcal G$.
\end{rem}

We finish the paper studying the existence of injective hulls. We
assume that our category $\mathcal A$ is abelian and that $\mathcal E$
is the abelian exact structure. Using the argument of Enochs and Xu in
\cite[$\S$2.2]{Xu} we will prove that in an exact substructure
$\mathcal F$, if an object $M$ is an $\mathcal F$-admissible subobject
of an $\mathcal F$-injective object, then $M$ actually has an
$\mathcal F$-injective hull. We shall need the hypothesis that
$\mathcal F$ is closed under well ordered limits. This condition is
stronger than being closed under transfinite compositions as the next
example shows.

\begin{expl} \rm
  Let $R$ be a non-noetherian countable ring. Then there exists an
  fp-injective module $M$ which is not injective. Consider the exact
  structure $\mathcal E^M$. Since $M$ is not injective, there exists an inclusion $u:I \rightarrow R$ which is not an
  $\mathcal E^M$-inflation. Since $I$ is countable,
  $I=\bigcup_{n < \omega}I_n$ for a chain of finitely generated right
  ideals of $R$. Now each inclusion $u_n:I_n \rightarrow R$ is an
  $\mathcal E^M$ inflation and the direct limit of all of them is
  $u$. Note that $\mathcal E^M$ is closed under transfinite
  compositions by Lemma \ref{l:TransfiniteCompositions}.
\end{expl}

\begin{lem}\label{l:epim}
  Suppose that $\mathcal A$ is an abelian category, $A$ is an object
  of $\mathcal A$ and $f,g \in \textrm{End}_{\mathcal A}(A)$ are two
  monomorphisms. If $\Img f \subseteq \Img fg$ then $g$ is epic.
\end{lem}

\begin{proof}
  Since $\mathcal A$ is abelian, each monomorphism is the kernel of
  its cokernel, so that the image $f$ and $fg$ are represented by the
  monomorphism $f$ and $fg$ respectively. The inclusion
  $\Img f \subseteq \Img fg$ as subobjects of $A$ implies that there
  exists a morphism $h:A \rightarrow A$ such that $f=fgh$. Then
  $f(1-gh)=0$ and, since $f$ is monic, $1-gh=0$. This implies that $g$
  is an epimorphism.
\end{proof}

Recall that an abelian category $\mathcal A$ is
said to satisfy AB5 if $\mathcal A$ is cocomplete and direct limits
are exact.

\begin{lem}\label{l:monomorphism}
  Suppose that $\mathcal A$ is an abelian category satisfying AB5. Let
  $\kappa$ be an ordinal,
  $(A_\alpha,u_{\beta\alpha})_{\alpha < \beta < \kappa}$ be a direct
  system of objects and
  $f: \displaystyle \lim_{\longrightarrow}A_\alpha \rightarrow A$ be a
  morphism. Suppose that for each $\alpha < \kappa$,
  $\Ker (fu_\alpha) = \Ker (u_{\alpha+1,\alpha})$, where
  $\displaystyle u_\alpha:A_\alpha \rightarrow
  \lim_{\longrightarrow}A_\gamma$ is the canonical morphism. Then $f$
  is a monomorphism.
\end{lem}

\begin{proof}
  Since $\mathcal A$ satisfies AB5, direct limits are exact and,
  consequently,
  $\displaystyle \Ker f=
  \lim_{\longrightarrow}\Ker(fu_\alpha)$. Denote
  $\Ker (fu_\alpha)=K_\alpha$ for each $\alpha < \kappa$. Since
  $K_\alpha$ is the kernel of $u_{\alpha+1,\alpha}$, we can construct,
  for each $\alpha < \kappa$, the following commutative diagram with exact rows
  \begin{displaymath}
    \begin{tikzcd}
      0 \arrow{r} & K_\alpha \arrow{r}{k_\alpha}
      \arrow{d}{k_{\alpha+1,n}} & A_\alpha
      \arrow{r}{u_{\alpha+1,\alpha}} \arrow{d}{u_{\alpha+1,\alpha}} &
      A_{\alpha+1} \arrow{r} \arrow{d}{u_{\alpha+1,\alpha+2}}& 0\\
      0 \arrow{r} & K_{\alpha+1} \arrow{r}{k_{\alpha+1}} &
      A_{\alpha+1} \arrow{r}{u_{\alpha+2,\alpha+1}} & A_{\alpha+2}
      \arrow{r} & 0
    \end{tikzcd}
  \end{displaymath}
  which actually defines a direct system of conflations. Taking direct limit and noting that
  $\displaystyle \lim_{\longrightarrow}u_{\alpha+1,\alpha}$ is the
  identity, the exactness of direct limits gives that
  $\displaystyle \lim_{\longrightarrow}K_\alpha=0$. Then $\Ker f=0$
  and $f$ is a monomorphism.
\end{proof}

Given an object $X$ of $\mathcal A$, recall that the \textit{comma
  category} $X \downarrow \mathcal A$ is the category whose class of
objects consists of all morphisms $f:X \rightarrow A$ with
$A \in \mathcal A$, and whose morphisms between two objects,
$u:X \rightarrow A$ and $v:X \rightarrow B$, are morphisms in
$\mathcal A$, $f:A \rightarrow B$, satisfying $fu=v$. Abusing
language, we shall denote the morphism between $u$ and $v$ by
$f:u \rightarrow v$ as well. Given a class
$\mathcal I$ of inflations of $\mathcal E$, we are going to denote by
$X \downarrow_{\mathcal I} \mathcal A$ the full subcategory of the
comma category $X \downarrow \mathcal A$ whose objects are all morphisms in $\mathcal I$.  We shall call an object $u$ of
$X \downarrow_{\mathcal I} \mathcal A$, a cogenerator if for any other
object $v$ of $X \downarrow_{\mathcal I} \mathcal A$, there exists a
morphism $f:v \rightarrow u$.

Recall that an abelian category is said to be \textit{locally small} if the class of subobjects of any object actually is a set.

\begin{lem}\label{l:ExistenceCogenerator}
  Let $\mathcal A$ be a locally small abelian category, $\mathcal E$
  the abelian exact structure and $\mathcal I$ a class of conflations of $\mathcal E$ which is closed under well ordered
  direct limits. Let $u$ be a cogenerator in
  $X \downarrow_{\mathcal I} \mathcal A$. Then there exists a
  cogenerator in $X \downarrow_{\mathcal I} \mathcal A$,
  $\overline u:X \rightarrow \overline E$, and a morphism
  $\overline f:u \rightarrow \overline u$ such that any morphism
  $f':\overline u \rightarrow u'$ in
  $X \downarrow_{\mathcal I} \mathcal A$ in which $u'$ is a
  cogenerator satisfies $\Ker(f'\overline f)=\Ker \overline f$.
\end{lem}

\begin{proof}
  Assume that the claim of the lemma is not true. We are going to
  construct, for each pair of ordinals $\beta < \alpha$, cogenerators
  $u_\beta$ and $u_\alpha$ and a morphism
  $f_{\alpha\beta}:u_\beta \rightarrow u_\alpha$ such that $u_0=u$,
  the system
  $(u_\gamma,f_{\gamma\delta})_{\delta < \gamma \leq \alpha}$ is
  directed, $\Ker f_{\beta 0} \subsetneq \Ker f_{\alpha 0}$. This is a
  contradiction since the category is locally small.

  We shall make the construction recursively on $\alpha$. For
  $\alpha=0$ let $u_0=u$. Let $\alpha > 0$ and assume that we have
  constructed $u_\delta$ and $f_{\delta\gamma}$ for each
  $\gamma < \delta < \alpha$. If $\alpha$ is successor, say
  $\alpha = \beta+1$, then as $u_\beta$ does not satisfy the claim of
  the lemma, there exists an inflation $u_{\beta+1}$ in
  $X \downarrow_{\mathcal I} \mathcal A$ and a morphism
  $f_{\beta+1\beta}:u_\beta \rightarrow u_{\beta+1}$ such that
  $\Ker f_{\beta 0} \subsetneq \Ker(f_{\beta+1\beta}f_{\beta
    0})$. Then set
  $f_{\beta+1 \delta}=f_{\beta+1\beta}f_{\beta \delta}$ for each
  $\delta < \alpha$. Clearly,
  $\Ker f_{\alpha 0} \subsetneq \Ker f_{\beta 0}$.

  If $\alpha$ is limit, set
  $\displaystyle u_\alpha = \lim_{\substack{\longrightarrow\\ \delta <
      \alpha}} u_\delta$ and
  $f_{\alpha \delta}:u_\delta \rightarrow u_\alpha$ the structural
  morphisms of this direct limit. By hypothesis, $u_\alpha$ is an element
  of $X \downarrow_{\mathcal I} \mathcal A$ which is a cogenerator, as
  each $u_\beta$ is a cogenerator for each $\beta < \alpha$. Moreover,
  $\Ker f_{\beta 0} \subsetneq \Ker f_{\alpha 0}$ for each
  $\beta < \alpha$ because, otherwise, if
  $\Ker f_{\beta 0} = \Ker f_{\alpha 0}$, then
  $\Ker f_{\beta+1 0} = \Ker f_{\alpha 0}$ as well, so that
  $\Ker f_{\beta 0} = \Ker f_{\beta+1 0}$, a contradiction. This
  finishes the proof.
\end{proof}

\begin{theorem}\label{t:ExistenceHulls}
  Let $\mathcal A$ be a locally small abelian category satisfying AB5, $\mathcal E$ be the abelian exact structure of $\mathcal A$,
  and $\mathcal I$ a class of conflations of
  $\mathcal E$ which is closed under well ordered limits. Let $X$ be
  any object of $\mathcal A$ such that there exists an
  inflation $u:X \rightarrow E$ in $\mathcal I$ with $E$, an
  $\mathcal I$-injective object. Then there exists an inflation $v:X \rightarrow F$ with $F$ an $\mathcal I$-injective object such that $v$ is minimal.
\end{theorem}

\begin{proof}
  Note that $u$ is a cogenerator in
  $X \downarrow_{\mathcal I} \mathcal A$ since $E$ is
  $\mathcal I$-injective. First of all, by setting $u_0=u$, we can
  apply recursively the preceding lemma to get, for each $n < \omega$,
  a cogenerator $u_n$ in $X \downarrow_{\mathcal I} \mathcal A$ and a
  morphism $f_{n+1,n}:u_n \rightarrow u_{n+1}$ such that any other
  morphism $f':u_{n+1} \rightarrow u'$ with $u'$ a cogenerator
  satisfies $\Ker f'f_{n+1,n} = \Ker f_{n+1,n}$. Let
  $\displaystyle w=\lim_{\substack{\longrightarrow\\n < \omega}} u_n$
  and note that $w$ is a cogenerator in
  $X \downarrow_{\mathcal I} \mathcal A$. Since any
  $f':w \rightarrow u'$ with $u'$ a cogenerator satisfies, for each
  natural number $n$, that $\Ker(f'f_n) = \Ker(f_{n+1,n})$, where
  $f_n$ is the canonical morphism of the direct limit, Lemma \ref{l:monomorphism} says that any such $f'$ is actually a monomorphism.

  Suppose that the cogenerator $w$ is of the form $w:X \rightarrow F$, and let us prove that $w$ is a minimal morphism, that is, that any
  $f:w \rightarrow w$ is an isomorphism. Let $f:w \rightarrow w$ be a
  morphism and assume that $f$ is not an isomorphism. Since it is
  monic, by the previous claim, $f$ is not an epimorphism. Now we can
  construct, by transfinite recursion, a monomorphism
  $f_{\alpha\beta}:w_\beta \rightarrow w_\alpha$ for each $\beta < \alpha$, where $w_\alpha=w$
  if $\alpha$ is successor and, otherwise,
  $\displaystyle w_\alpha = \lim_{\substack{\longrightarrow\\ \gamma <
      \alpha}}w_\gamma$, such that $f_{\alpha\beta} = f$ if
  $\alpha = \beta+1$. Cases $\alpha=0$ and $\alpha$, a successor are
  easy. If $\alpha$ is a limit ordinal, set
  $\displaystyle w_\alpha = \lim_{\substack{\longrightarrow\\ \gamma <
      \alpha}} w_\gamma$ with structural maps
  $f'_{\alpha\gamma}:w \rightarrow w_\alpha$ for each
  $\gamma < \alpha$. Since $w$ is a cogenerator, there exists
  $f'_\alpha:w_\alpha \rightarrow w$. Then define
  $f_{\alpha\beta}=f'_\alpha f'_{\alpha \beta}$.

  Now we prove that for each ordinal $\alpha$,
  $\{\Img f_{\alpha\beta}:\beta < \alpha+1\}$ is a strictly ascending
  chain of subobjects of $F$, which is a contradiction. Take
  $\beta < \alpha+1$ and suppose that
  $\Img f_{\beta+1,\alpha+1} = \Img f_{\beta,\alpha+1}$. Since
  $f_{\beta,\alpha+1} = f_{\beta+1,\alpha+1}f$, Lemma \ref{l:epim}
  implies that $f_{\beta+1,\alpha+1}$ is an epimorphism. But
  $f_{\beta+1,\alpha+1} = ff_{\beta+1,\alpha}$ so that $f$ is an
  epimorphism as well. This contradicts the previous hypothesis and $f$ has to be an isomorphism.
\end{proof}

\begin{cor}
  Let $\mathcal A$ be a locally small abelian category satisfying AB5, $\mathcal E$ be the abelian exact structure of $\mathcal A$,
  and $\mathcal F$ an additive exact substructure of
  $\mathcal E$ which is closed under well ordered limits. Let $X$ be
  any object of $\mathcal A$ such that there exists an
  $\mathcal F$-inflation $u:X \rightarrow E$ with $E$, an
  $\mathcal F$-injective object. Then $X$ has a $\mathcal F$-injective envelope.
  Moreover, this $\mathcal F$-injective
  envelope is an $\mathcal F$-injective hull as well.
\end{cor}

\begin{proof}
Follows immediately from the previous result. By Theorem \ref{t:InjectiveHulls}, every $\mathcal F$-injective envelope actually is a $\mathcal F$-injective hull.
\end{proof}

\section{Applications}

\noindent In this section we give several applications of the results obtained in the previous sections.

\subsection{Approximations in exact categories}

In the recent years several papers studying approximations in exact categories have appeared in the literature. There are two ways of defining approximations in a category $\mathcal D$. The first of them takes a fixed class of objects $\mathcal X$ and is based on the notions of $\mathcal X$-preenvelope and $\mathcal X$-precover defined at the beginning of Section 3. These are the approximations widely studied for module categories and the ones extended in \cite{SaorinStovicek11} to exact categories.

The other way of defining approximations takes an ideal $\mathcal I$ in the category (that is, a subfunctor of the $\Hom_{\mathcal D}$ bifunctor) and is based on the notion of $\mathcal I$-preenvelopes and $\mathcal I$-precovers (recall that a $\mathcal I$-preenvelope of an object $D$ of $\mathcal D$ is a morphism $i:D \rightarrow X$ that belongs to $\mathcal I$ and such that for any other morphism $j:D \rightarrow Y$, there exists $f:X \rightarrow Y$ with $fj=i$; the $\mathcal I$-precovers are defined dually). This is the approach of \cite{FuGuilHerzogTorrecillas}.

In this paper, we are going to apply the results of the previous sections in the study of approximations by objects. As a direct consequence of Theorem \ref{t:ExistenceInjectives} we get that the class of injective objects with respect to certain sets of inflations provide for preenvelopes. We shall use the following lemma for the
exact structure $\mathcal E^{\mathcal X}$ where $\mathcal X$ is a
class of objects.

\begin{lem}\label{l:TransfiniteCompositions}
  Suppose that transfinite compositions of inflations in $\mathcal E$ exist and are inflations and let $\mathcal X$ be any class of objects. Then transfinite compositions of
  $\mathcal E^{\mathcal X}$-inflations exist and are $\mathcal E^{\mathcal X}$-inflations.
\end{lem}

\begin{proof}
  Let $(Y_\alpha,u_{\beta\alpha})_{\alpha < \beta < \kappa}$ be a
  direct system of objects indexed by an ordinal $\kappa$, such that $u_{\beta\alpha}$ is an
  $\mathcal E^{\mathcal X}$-inflation for each $\alpha < \beta < \kappa$. Denote by
  $\displaystyle u_\alpha:Y_\alpha \rightarrow
  \lim_{\longrightarrow}Y_\beta$ the canonical morphism for each
  $\alpha < \kappa$. Given any $X \in \mathcal X$ and any
  $f:Y_0 \rightarrow X$ we can construct, using that $u_{\beta\alpha}$ is an
  $\mathcal E^{\mathcal X}$-inflation for each $\alpha < \beta < \kappa$, a direct system of morphisms, $(f_\alpha:Y_\alpha \rightarrow X)_{\alpha < \kappa}$, with $f_0 = f$. Then the induced morphism $g:\varinjlim_{\beta < \kappa}Y_\beta \rightarrow X$ satisfies $gu_0 = f$. This means that $u_0$ is an $\mathcal E^{\mathcal X}$-inflation.
\end{proof}

\begin{cor}
Suppose that transfinite compositions of inflations in $\mathcal A$ exist and are inflations. Let $\mathcal H$ be a set of inflations such that for each $i:K \rightarrow H$ in $\mathcal H$, $K$ is small. Let $\mathcal X$ be the class of all $\mathcal H$-injective objects. Then each object in $\mathcal A$ has a $\mathcal X$-preenvelope.
\end{cor}

\begin{proof}
By Lemma \ref{l:TransfiniteCompositions}, the exact category $(\mathcal A; \mathcal E^{\mathcal X})$ satisfy the hypotheses of Theorem \ref{t:ExistenceInjectives}, since an object is $\mathcal E^{\mathcal X}$-injective if and only if it is $\mathcal H$-injective. Then, for each object $A$ of $\mathcal A$, there exists an $\mathcal E^{\mathcal X}$-inflation $i:A \rightarrow E$ with $E$ an $\mathcal E^{\mathcal X}$-injective object. But any $\mathcal E^{\mathcal X}$-injective object actually belongs to $\mathcal X$ (since morphisms in $\mathcal H$ are $\mathcal E^{\mathcal X}$-inflations), so that $i$ is a $\mathcal X$-preenvelope.
\end{proof}

One situation that fits the hypotheses of the preceding result is when we take, in the category of right modules over a unitary ring $R$,
$\mathcal H$ to be the set of all conflations
$K \rightarrow R^n \rightarrow L$ with $n$ a natural number and $K$
finitely generated. In this case, the class $\mathcal X$ consists of
all fp-injective modules. The preceding results gives that every
module has an fp-injective preenvelope (see \cite[Theorem
4.1.6]{GobelTrlifaj}).

\begin{cor} \label{fp-inj}
  Let $R$ be a ring. Then every module has an fp-injective preenvelope.
\end{cor}

Maybe, the most general result regarding approximations in exact categories is \cite[Theorem 2.13(4)]{SaorinStovicek11}. We see that this result can be deduced from our Theorem \ref{t:ExistenceInjectives}.

Let $\mathcal I$ be a set of inflations and denote by $\Coker(\mathcal I)$ the class consisting of all cokernels of morphism in $\mathcal I$. Recall that $\mathcal I$ is called
\textit{homological} \cite[Definition 2.3]{SaorinStovicek11} if the
following two conditions are equivalent for any object $T$:
\begin{enumerate}
\item $T \in \Coker(\mathcal I)^\perp$.

\item $T$ is $\mathcal I$-injective.
\end{enumerate}

\begin{cor}
  Suppose that transfinite compositions of inflations exist and are
  inflations. Let $\mathcal I$ be a homological set of inflations such
  that, for each $i:K \rightarrow L$ in $\mathcal I$, $K$ is
  small. Then $\Coker(\mathcal I)^\perp$ is preenveloping.
\end{cor}

\begin{proof}
  Denote $\Coker(\mathcal I)^\perp$ by $\mathcal S$. Note that, since $\mathcal I$ is homological, an object belongs to $\mathcal S$ if
  and only if it is $\mathcal I$-injective, if and only if it is $\mathcal E^{\mathcal S}$-injective, so that the
  result is equivalent to prove that there exists enough
  $\mathcal E^{\mathcal S}$-injectives. But by Lemma
  \ref{l:TransfiniteCompositions}, transfinite compositions of
  $\mathcal E^{\mathcal S}$-inflations exist and are $\mathcal E^{\mathcal S}$-inflations and an object is
  $\mathcal E^{\mathcal S}$-injective (equivalently, belongs to
  $\mathcal S$) if and only if it is $\mathcal I$-injective. Then the
  existence of $\mathcal E^{\mathcal S}$-injectives follows from
  Theorem \ref{t:ExistenceInjectives}.
\end{proof}

\subsection{Approximations in Grothendieck categories}

In this subsection, $\mathcal D$ will be a Grothendieck category with the abelian exact structure. The first application of our results is the existence of injective hulls in $\mathcal D$.

\begin{cor} \label{Grothendieck}
Every object in the Grothendieck category $\mathcal D$ has an injective hull.
\end{cor}

\begin{proof}
First we show that $\mathcal D$ satisfies the hypotheses of Theorem \ref{t:ExistenceInjectives} to prove that $\mathcal D$ has enough injectives. That transfinite
  composition of inflations are inflations follows from (AB5),
  \cite[Proposition V.1.1]{Stenstrom}. In order to see that $\mathcal D$
  satisfies (2) of Theorem \ref{t:ExistenceInjectives}, we shall see
  that it satisfies the conditions of Remark \ref{r:Hyphoteses}. First note
  that $\mathcal D$ is locally small by \cite[Proposition
  IV.6.6]{Stenstrom}. On the other hand, it is well known that all
  objects in a Grothendieck category are small and, if $G$ is a
  generator of $\mathcal D$, an object is $G$-injective if and only if it is
  injective by \cite[Proposition V.2.9]{Stenstrom}. Consequently, $\mathcal D$ has enough injectives by Theorem \ref{t:ExistenceInjectives}.
  
  Now, the existence of injective hulls in $\mathcal D$ is a direct consequence of Theorem \ref{t:ExistenceHulls}.
\end{proof}

Now, let us look at approximations in $\mathcal D$ by a class of objects. In many situations, the classes providing for approximations belong to a cotorsion pair. The relationship between cotorsion pairs and approximations in module categories was first observe by Salce in the late 1970s who proved that, if $(\mathcal B,\mathcal C)$ is a cotorsion pair, then the existence of special $\mathcal B$-precovers is equivalent to the existence of special $\mathcal C$-preenvelopes (a special $\mathcal B$-precover of a module $M$ is a morphism $f:B \rightarrow M$ with $B \in \mathcal B$ and $\Ker f \in \mathcal B^\perp$; the special $\mathcal C$-preenvelopes are defined dually). Later on, Enochs proved the important fact that a \textit{closed} (in the sense that $\mathcal B$ is closed under direct limits) and complete cotorsion pair provide minimal approximations: covers and envelopes. Finally, Eklof and Trlifaj proved that complete cotorsion pair are abundant: any cotorsion pair generated by a set is complete. All these works were motivated by the study of the existence of flat covers, the so-called ``Flat cover conjecture", solved by Bican, El Bashir and Enochs in \cite{BicanBashirEnochs}.

In this section we see that these results are consequences of our results in the previous sections. Recall that for a class $\mathcal X$ of objects we can form the cotorsion pair $(\mathcal X^\perp, {^\perp}(\mathcal X^\perp))$, which is called the cotorsion pair generated by $\mathcal X$. We say that a cotorsion pair $(\mathcal B,\mathcal C)$ is \textit{generated by a set} if there exists a set of objects $\mathcal S$ such that $(\mathcal B,\mathcal C)$ coincides with the cotorsion pair generated by $\mathcal S$. Moreover, we say that $(\mathcal B, \mathcal C)$ is \textit{closed} if $\mathcal B$ is closed under direct limits.

\begin{theorem}
Let $(\mathcal B,\mathcal C)$ be a cotorsion pair in $\mathcal D$.
\begin{enumerate}
\item If $\mathcal D$ has a projective generator and $(\mathcal B,\mathcal C)$ is cogenerated by a set, then $(\mathcal B,\mathcal C)$ is complete.

\item If $(\mathcal B,\mathcal C)$ is complete and closed then every object has a $\mathcal C$-envelope.
\end{enumerate}
\end{theorem} 

\begin{proof}
(1) We prove, using Theorem \ref{t:ExistenceInjectives}, that the exact structure $\mathcal E^{\mathcal C}$ has enough injective objects. First note that transfinite compositions of inflations are inflations by \cite[Proposition V.1.1]{Stenstrom}. Using Lemma \ref{l:TransfiniteCompositions} we deduce that transfinite compositions of $\mathcal E^{\mathcal C}$-inflations are $\mathcal E^{\mathcal C}$-inflations as well. Now, since $\mathcal D$ has a projective generator, for each $S \in \mathcal S$ there exists a conflation
\begin{displaymath}
\begin{tikzcd}
K_S \arrow{r}{i_S} & P_S \arrow{r}{p_S} & S
\end{tikzcd}
\end{displaymath}
with $P_S$ projective. Let $\mathcal H$ be the set $\{i_S:S \in \mathcal S\}$. As $\mathcal D$ is a Grothendieck category, $K_S$ is small for each $S \in \mathcal S$. Moreover, it is easy to show that $\mathcal H$ is contained in $\mathcal E^{\mathcal C}$, which implies that an object $M$ is $\mathcal E^{\mathcal C}$-injective if and only if it is $\mathcal H$-injective. We can apply Theorem \ref{t:ExistenceInjectives} to get that $\mathcal E^{\mathcal C}$ has enough injective objects. Then, noting that an object $M$ is $\mathcal H$-injective if and only if $M \in \mathcal C$, we conclude that any $\mathcal E^{\mathcal C}$-inflation $i:M \rightarrow E$ with $E$ a $\mathcal E^{\mathcal C}$-injective object actually is a $\mathcal C$-preenvelope. Consequently, $\mathcal C$ is preenveloping.

Now let us take the $\mathcal C$-preenvelope $i:M \rightarrow E$ of an object $M$ as constructed in Theorem \ref{t:ExistenceInjectives}. By Remark \ref{r:StructurePreenvelope}, there exists an infinite regular cardinal $\kappa$ and a $\kappa$-sequence $(P_\alpha,f_{\beta\alpha})_{\alpha<\beta<\kappa}$ such that $P_0=M$, $i$ is the transfinite composition of the sequence and, for each $\alpha < \kappa$, $f_{\alpha+1,\alpha}$ is a pushout of a direct sum of inflations belonging to $\mathcal H$. This last condition implies that $\Coker f_{\alpha+1,\alpha}$ is a direct sum of modules belonging to $\mathcal S$ and, consequently, belongs to $\mathcal B$. Now, for each $\alpha < \beta < \kappa$ we get the commutative diagram of conflations
\begin{displaymath}
\begin{tikzcd}
M \arrow{r}{f_{\alpha 0}} \arrow{d} & P_\alpha \arrow{r}{\overline f_{\alpha 0}} \arrow{d}{f_{\beta\alpha}} & \Coker f_{\alpha 0} \arrow{d}{\overline f_{\beta\alpha}}\\
M \arrow{r}{f_{\beta 0}} & P_\beta \arrow{r}{\overline f_{\beta 0}} & \Coker f_{\beta 0}\\
\end{tikzcd}
\end{displaymath}
whose direct limit is the conflation
\begin{displaymath}
\begin{tikzcd}
M \arrow{r}{i} & E \arrow{r}{p} & \Coker i
\end{tikzcd}
\end{displaymath}
In particular, we get that $\Coker i$ is the composition of the transfinite sequence $(\Coker f_{\alpha 0}, \overline f_{\beta\alpha})_{\alpha < \beta < \kappa}$. Using the snake lemma it is easily verified that $\Coker \overline{f}_{\alpha+1,\alpha} \cong \Coker f_{\alpha+1,\alpha} \in \mathcal C$, so that, by Eklof lemma \cite[Proposition 2.12]{SaorinStovicek11}, $\Coker i \in \mathcal B$ as well. This means that $i$ is a special $\mathcal C$-preenvelope and that the cotorsion pair is complete.

(2) Let $M$ be an object in $\mathcal D$. Using that the cotorsion pair is complete, there exists an inflation $i_M:M \rightarrow E$ with $E \in \mathcal C$ and $\Coker i_M \in \mathcal B$. Now denote by $\mathcal I$ the class of inflations $i:A \rightarrow B$ in $\mathcal E$ such that $\Coker i \in \mathcal  B$. Since $\mathcal B$ and $\mathcal E$ are closed under direct limits, then so is $\mathcal I$. Moreover, notice that $i_M \in \mathcal I$ and satisfies, by Corollary \ref{c:CotorsionPair}, that $E$ is $\mathcal I$-injective. Then we are in position to apply Theorem \ref{t:ExistenceHulls} to obtain a minimal inflation $j_M:M \rightarrow F$ in $\mathcal I$ with $\mathcal F$ an $\mathcal I$-injective object. Using that the cotorsion pair is complete there exists a conflation
\begin{displaymath}
\begin{tikzcd}
E \arrow{r}{u} & C \arrow{r} & B
\end{tikzcd}
\end{displaymath} 
with $C \in \mathcal C$ and $B \in \mathcal B$. In particular, $u \in \mathcal I$ and, as $E$ is $\mathcal I$-injective, this conflation is split. This means that $E \in \mathcal C$. Consequently, the inflation $j_M$ actually is a $\mathcal C$-envelope.
\end{proof}

\subsection{Pure-injective hulls in finitely accessible additive categories}

The notion of purity in module categories can be considered in general in finitely accessible additive categories. Let $\mathcal K$ be an additive category with direct limits. Recall that an object $F$ of $\mathcal K$ is \textit{finitely presented} if for each direct system of objects in $\mathcal K$, $(K_i,u_{ji})_{i < j \in I}$, the canonical morphism from $\varinjlim \Hom_{\mathcal K}(F,X_i) \rightarrow \Hom_{\mathcal K}\left(F,\varinjlim K_i\right)$ is an isomorphism. The category $\mathcal K$ is said to be finitely accessible if it has all direct limits and there exists a set $\mathcal S$ of finitely presented objects such that every object of $\mathcal K$ can be expressed as a direct limit of objects from $\mathcal S$.

Let $\mathcal K$ be a finitely accessible additive category. A kernel-cokernel pair in $\mathcal K$
\begin{displaymath}
\begin{tikzcd}
K \arrow{r}{i} & M \arrow{r}{p} & L
\end{tikzcd}
\end{displaymath}
is said to be pure if for each finitely presented module $P$, the sequence of abelian groups
\begin{displaymath}
\begin{tikzcd}
\Hom_{\mathcal K}(P,K) \arrow{r} & \Hom_{\mathcal K}(P,M) \arrow{r} & \Hom_{\mathcal K}(P,L)
\end{tikzcd}
\end{displaymath}
is exact. The class $\mathcal E_{\mathcal P}$ of all such kernel-cokernels pairs is an exact structure on $\mathcal K$, which we shall call the pure-exact structure. As in the case of modules, inflations, deflations, injectives and projectives with respect to this exact structure will be called pure-monomorphisms, pure-epimorphisms, pure-injectives and pure-projectives respectively. The main objective of this section is to apply the results of the previous one to study the existence of pure-injective hulls in $\mathcal K$.

It was proved in \cite{Crawley} that every finitely accessible additive category is equivalent to the full subcategory Flat-$\mathcal S$ of additive flat functors from a small preadditive category $\mathcal S$ to the category of abelian groups. Using that any functor category (additive functors from a small preadditive category to the category of abelian groups) is equivalent to the category of unitary modules over a ring $T$ with enough idempotents (that is, an associative ring without unit but with a family of pairwise orthogonal elements $\{e_i\mid i \in I\}$ such that $T=\oplus_{i \in I}Te_i = \oplus_{i \in I}e_iT$), we get that any finitely accessible additive category is equivalent to the full subcategory of flat modules over a ring with enough idempotents. More precisely, if $\mathcal K$ is a finitely accessible additive category and $\{F_i: i \in I\}$ is a representing set of the isomorphism classes of finitely presented objects of $\mathcal K$, and we denote by $F=\oplus_{i \in I}F_i$, then we consider $T=\widehat{\End}_\mathcal{K}(F)$ the subring of $\End_{\mathcal K}(F)$ consisting of all endomorphisms $f$ of $F$ such that $f(F_i)=0$ except for possibly finitely many indices $i \in I$. This ring, called the \textit{functor ring} of the family $\{F_i:i \in I\}$, is a ring with enough idempotents such that $\mathcal K$ is equivalent to the full subcategory of Mod-$T$ consisting of flat modules, in such a way that pure exact sequences in $\mathcal K$ corresponds to exact sequences in Flat-$R$.

So, in order to study a finitely accessible additive category we can restrict ourselves to the full subcategory of flat modules over a ring with enough idempotents. 


The following result can be proved as in the unitary case (see \cite[Lemma 5.3.12]{EnochsJenda}):

\begin{lem}
Let $T$ be a ring with enough idempotents with $|T| = \kappa$. For each unitary $T$-module $M$ and element $x \in M$ there exists a pure submodule $S$ of $M$ containing $x$ such that $|S| \leq \kappa$.
\end{lem}

Now we prove that any accessible category satisfies the Baer's criterion.

\begin{theorem}
Let $\mathcal K$ be a finitely accessible additive category. There exists a cardinal number $\kappa$ such that if $\mathcal G$ is the set of objects
\begin{displaymath}
\left\{\bigoplus_{i \in I}G_i\mid G_i \textrm{ is finitely presented and }|I| \leq \kappa\right\}
\end{displaymath} 
then any $\mathcal G$-pure-injective object is pure-injective.
\end{theorem}

\begin{proof}
As mentioned before, we may assume that $\mathcal K$ is the full subcategory consisting of unitary flat right $T$-modules over a ring $T$ with enough idempotents. Let $M$ be a flat $T$-module which is $\mathcal G$-pure-injective. In order to see that it is pure-injective we
  only have to see that it is pure-injective with respect to all
  direct sums of finitely presented modules by Lemma
  \ref{l:AInjective} and \cite[33.5]{Wisbauer}.

  Let $I$ be a set, $\{F_i:i \in I\}$ a family of finitely presented
  modules in $\mathcal K$ (that is, finitely generated and projective in Mod-$T$), $K$ a pure submodule of $\bigoplus_{i \in I}F_i$ and
  $f:K \rightarrow M$ a morphism. Denote by $|K| = \lambda$. Using the preceding lemma, we can construct a chain of subsets of $I$,
  $\{I_\alpha:\alpha < \lambda\}$, satisfying, for each $\beta < \lambda$ that
  $\bigcup_{\alpha < \lambda}I_\alpha=I$,
  $|I_{\beta+1}-I_\beta| \leq \kappa$ and
  $I_\beta = \bigcup_{\alpha < \beta}I_\alpha$ (when $\beta$ is limit); and a chain of pure submodules of $K$,
  $\{K_\alpha:\alpha < \lambda\}$, satisfying, for each $\beta < \lambda$, that
  $K=\bigcup_{\alpha < \lambda}K_\alpha$, $K_\beta \leq \bigoplus_{i \in I_\beta}F_i$ and
  $K_\beta = \bigcup_{\alpha < \beta}K_\alpha$ (when $\beta$ is limit).

  Now, using that $M$ is $\mathcal G$-injective, we can define,
  recursively on $\alpha$, a morphism
  $f_\alpha:\bigoplus_{i \in I_\alpha}F_i \rightarrow M$ such that
  $f_\alpha \rest K_\alpha = f \rest K_\alpha$ and
  $f_\alpha \rest \bigoplus_{i \in I_\beta}F_i = f_\beta$ for each
  $\beta < \alpha$. Then the limit of all these $f_\alpha$'s is the
  extension of $f$ to $\bigoplus_{i \in I}F_i$. This finishes the
  proof.
\end{proof}

Then we get:

\begin{cor} \label{pinj}
Let $\mathcal K$ be a finitely accessible additive category. Then $\mathcal K$ has enough pure-injective objects. If, in addition, $\mathcal K$ is abelian, then $\mathcal K$ has pure-injective hulls.
\end{cor}

\begin{proof}
Again, we can assume that $\mathcal K$ is the full subcategory consisting of unitary flat right $T$-modules over a ring $T$ with enough idempotents. Since direct limits of conflations in Mod-$T$ are conflations and $\mathcal K$ is closed under direct limits in Mod-$T$, direct limits of pure-exact sequences in $\mathcal K$ are, again, pure-exact. Then $\mathcal K$ satisfies (1) of
  Theorem \ref{t:ExistenceInjectives}. Let $\mathcal G$ be the set of objects constructed in the previous result and let us see that $\mathcal G$
  satisfies the conditions of the Remark \ref{r:Hyphoteses}. Since Mod-$T$ is
  locally small and each module is small, $\mathcal G$ satisfies (1) and
  (2) of Remark \ref{r:Hyphoteses}. Moreover, it satisfies (3) by the preceding theorem. By Theorem \ref{t:ExistenceInjectives}, $\mathcal K$ has enough injective objects.

If $\mathcal K$ is abelian, then the existence of pure-injective hulls follows from Theorem \ref{t:ExistenceHulls}.
\end{proof}

\bigskip

\bibliographystyle{abbrv}

\begin{thebibliography}{10}

\bibitem{AdamekHerrlichStrecker}
J.~Ad\'{a}mek, H.~Herrlich, and G.~E. Strecker.
\newblock {\em Abstract and concrete categories: the joy of cats}.
\newblock Number~17. 2006.
\newblock Reprint of the 1990 original [Wiley, New York; MR1051419].

\bibitem{AdamekRosicky}
J.~Ad{\'a}mek and J.~Rosick{\'y}.
\newblock {\em Locally presentable and accessible categories}, volume 189 of
  {\em London Mathematical Society Lecture Note Series}.
\newblock Cambridge University Press, Cambridge, 1994.

\bibitem{BicanBashirEnochs}
L.~Bican, R.~El~Bashir, and E.~Enochs.
\newblock All modules have flat covers.
\newblock {\em Bull. London Math. Soc.}, 33(4):385--390, 2001.

\bibitem{Buhler}
T.~B{\"u}hler.
\newblock Exact categories.
\newblock {\em Expo. Math.}, 28(1):1--69, 2010.

\bibitem{Crawley}
W.~Crawley-Boevey.
\newblock Locally finitely presented additive categories.
\newblock {\em Comm. Algebra}, 22(5):1641--1674, 1994.

\bibitem{Enochs} E.~E. Enochs.
\newblock  Injective and flat covers and resolvents.
\newblock {\em Israel J. Math.} 39:189-209, 1981.

\bibitem{EnochsJenda}
E.~E. Enochs and O.~M.~G. Jenda.
\newblock {\em Relative homological algebra}, volume~30 of {\em de Gruyter
  Expositions in Mathematics}.
\newblock Walter de Gruyter \& Co., Berlin, 2000.

\bibitem{FuGuilHerzogTorrecillas}
X.~H. Fu, P.~A. Guil~Asensio, I.~Herzog, and B.~Torrecillas.
\newblock Ideal approximation theory.
\newblock {\em Adv. Math.}, 244:750--790, 2013.

\bibitem{Gnacadja}
G.~P. Gnacadja.
\newblock Phantom maps in the stable module category.
\newblock {\em J. Algebra}, 201(2):686--702, 1998.

\bibitem{GobelTrlifaj}
R.~G{\"o}bel and J.~Trlifaj.
\newblock {\em Approximations and endomorphism algebras of modules. {V}olume
  1}, volume~41 of {\em de Gruyter Expositions in Mathematics}.
\newblock Walter de Gruyter GmbH \& Co. KG, Berlin, extended edition, 2012.
\newblock Approximations.

\bibitem{GomezGuil}
J.~L. G\'omez~Pardo and P.~A. Guil~Asensio.
\newblock Chain conditions on direct summands and pure quotient modules.
\newblock In {\em Interactions between ring theory and representations of
  algebras ({M}urcia)}, volume 210 of {\em Lecture Notes in Pure and Appl.
  Math.}, pages 195--203. Dekker, New York, 2000.

\bibitem{Herzog}
I.~Herzog.
\newblock The phantom cover of a module.
\newblock {\em Adv. Math.}, 215(1):220--249, 2007.

\bibitem{Keller}
B.~Keller.
\newblock Chain complexes and stable categories.
\newblock {\em Manuscripta Math.}, 67(4):379--417, 1990.

\bibitem{MacGibbon}
C.~A. McGibbon.
\newblock Phantom maps.
\newblock In {\em Handbook of algebraic topology}, pages 1209--1257.
  North-Holland, Amsterdam, 1995.

\bibitem{Monari}
E.~Monari~Martinez.
\newblock On pure-injective modules.
\newblock In {\em Abelian groups and modules ({U}dine, 1984)}, volume 287 of
  {\em CISM Courses and Lect.}, pages 383--393. Springer, Vienna, 1984.

\bibitem{Prest09}
M.~Prest.
\newblock {\em Purity, spectra and localisation}, volume 121 of {\em
  Encyclopedia of Mathematics and its Applications}.
\newblock Cambridge University Press, Cambridge, 2009.

\bibitem{SaorinStovicek11}
M.~Saor{\'{\i}}n and J.~{\v{S}}{\v{t}}ov{\'{\i}}{\v{c}}ek.
\newblock On exact categories and applications to triangulated adjoints and
  model structures.
\newblock {\em Adv. Math.}, 228(2):968--1007, 2011.

\bibitem{Stenstrom}
B.~Stenstr{\"o}m.
\newblock {\em Rings of quotients}.
\newblock Springer-Verlag, New York, 1975.
\newblock Die Grundlehren der Mathematischen Wissenschaften, Band 217, An
  introduction to methods of ring theory.

\bibitem{Warfield}
R.~B. Warfield, Jr.
\newblock Purity and algebraic compactness for modules.
\newblock {\em Pacific J. Math.}, 28:699--719, 1969.

\bibitem{Wisbauer}
R.~Wisbauer.
\newblock {\em Grundlagen der {M}odul- und {R}ingtheorie}.
\newblock Verlag Reinhard Fischer, Munich, 1988.
\newblock Ein Handbuch f{\"u}r Studium und Forschung. [A handbook for study and
  research].

\bibitem{Xu}
J.~Xu.
\newblock {\em Flat covers of modules}, volume 1634 of {\em Lecture Notes in
  Mathematics}.
\newblock Springer-Verlag, Berlin, 1996.

\bibitem{Ziegler}
M.~Ziegler.
\newblock Model theory of modules.
\newblock {\em Ann. Pure Appl. Logic}, 26(2):149--213, 1984.

\end{thebibliography}

\bigskip

\bigskip

\end{document}